\documentclass[11pt,reqno]{amsart}
\usepackage[left=1.2in, right=1.2in]{geometry}

\usepackage{amsmath}
\usepackage{amsfonts}
\usepackage{mathrsfs}
\usepackage{amssymb}
\usepackage{amsthm}
\usepackage{enumitem}
\usepackage{bbm}
\usepackage{times}
\usepackage{tabularx}
\usepackage{amsbsy}
\usepackage{mathtools}
\usepackage{tikz}
\usetikzlibrary{arrows,decorations.markings}
\usetikzlibrary{matrix}
\usetikzlibrary{graphs}
\usetikzlibrary{backgrounds}
\usepackage{setspace}\spacing{1}
\usepackage{color}

\usepackage[colorlinks=true,linkcolor=blue,citecolor=blue]{hyperref}

\theoremstyle{Theorem}
\newtheorem{theorem}{Theorem} [section]

\newtheorem{proposition}[theorem]{Proposition}
\newtheorem{claim}[theorem]{Claim}
\newtheorem{lemma}[theorem]{Lemma}
\newtheorem{corollary}[theorem]{Corollary}

\newtheorem{definition}[theorem]{Definition}

\newtheorem{remark}[theorem]{Remark}

\theoremstyle{remark}

\newlist{enumlemma}{enumerate}{3}
\setlist[enumlemma]{label*={(\alph*)}, ref= {(\alph*)} }

\usepackage{marginnote}

\newcommand{\diff}{\mathrm{Diff}}

\renewcommand{\epsilon}{\varepsilon}

\DeclareMathOperator{\supp}{supp}
\DeclareMathOperator{\Diff}{Diff}

\newcommand{\R}{\mathbb {R}}
\newcommand{\Q}{\mathbb {Q}}
\newcommand{\Z}{\mathbb {Z}}
\newcommand{\N}{\mathbb {N}}
\newcommand{\T}{\mathbb {T}}

\newcommand{\s}{\mathbb {S}}

\newcommand{\e}{\epsilon}

 \newcommand{\oldepsilon}{\mathchar"10F}
\newcommand{\eps}{\oldepsilon}

\newcommand{\Sr}{S_{\text{rot}}}

\def\bu{\mathbf{u}}

\def\rot{\mathrm{rot}}
\def\FF{\mathcal{F}}
\def\SL{\mathrm{SL}}

\title[global rigidity of ABC group actions]{ Global Rigidity of Some Abelian-by-Cyclic group actions on $\T^2$}

\author[Sebastian Hurtado]{Sebastian Hurtado}
\address{Department of Mathematics, University of Chicago, Chicago, IL 60637, USA}
\email{shurtados@uchicago.edu}

\author[Jinxin Xue]{Jinxin xue}
\address{Yau Mathematical Sciences Center \& Department of Mathematics, Tsinghua University, Beijing, China, 100084}
\email{jxue@tsinghua.edu.cn}

\long\def\symbolfootnote[#1]#2{\begingroup\def\thefootnote{\fnsymbol{footnote}}
\footnote[#1]{#2}\endgroup}

\begin{document}

\begin{abstract}
For groups of diffeomorphisms of $\T^2$ containing an Anosov diffeomorphism, we give a complete classification for polycyclic Abelian-by-Cyclic group actions on $\T^2$ up to both topological conjugacy and smooth conjugacy under mild assumptions. Along the way, we also prove a Tits alternative type theorem for some groups of diffeomorphisms of $\T^2$.
 
\end{abstract}

\maketitle

\begin{section}{Introduction}

In this article, we study the action of certain abelian-by-cyclic groups (ABC, for simplicity) on $\T^2$. Our groups are given as follows. Let $B \in \text{SL}_n(\Z)$ be an invertible matrix with integer coefficients and determinant one, and let $\Gamma_B = \Z \ltimes_B \Z^n$ be the semi-direct product, where $\Z$ acts on $\Z^n$ via  the linear standard action of $B$ on $\Z^n$. 

\subsection{The topological classification result}
 We consider the case where $n = 2$, that is $B \in \SL_2(\Z)$ and $\Gamma_B = \Z \ltimes_B \Z^2$. We define an \emph{affine action} of $\Gamma_B$ on $\T^2$ to be a group homomorphism $\Phi:\  \Gamma_B\to \mathrm{Aff}(\T^2)$, where $\mathrm{Aff}(\T^2)$ is the group of orientation preserving affine transformations on $\T^2$, and \begin{equation}\label{EqAffine} \Phi(B)=A\in \SL_2(\Z),\quad \Phi(e_i): x\mapsto x+\rho_i\ \mathrm{mod}\ \Z^2,\ i=1,2,\end{equation}
where $e_1=(1,0)$, $e_2=(0,1)\in \Z^2$ and $\rho_i\in \R^2,\ i=1,2$.

In this case, we let $\rho=(\rho_1,\rho_2)$ to be the matrix having $\rho_1$ and $\rho_2$ as columns and call it the rotation matrix of $\Phi$.  The group relation $\Phi(B)\Phi(v)=\Phi(Bv)\Phi(B),\ \forall \ v\in \Z^2,$ is equivalent to 
\begin{equation}\label{EqCommute}
A\rho=\rho B,\  \mathrm{mod\ }\Z^{2\times 2},
\end{equation}
 in other words, for the equation \label{EqAffine} to define an action, there must exists a matrix $C\in \Z^{2\times 2}$ such that $A\rho=\rho B+C$. As discussed in Section \ref{SSHomotopy} (Corollary \ref{ABCeasy}), these actions can only be faithful in the case where $|tr A| = |tr B|$. Affine actions will be classified in Section \ref{SSAffine}.

One of our main results is the following global rigidity result for actions of  $\Gamma_B$ on $\T^2$ by diffeomorphisms. We denote by $\Diff^r(\T^2)$ the group of $C^r$ diffeomorphisms on $\T^2$ and by  $\mathrm{Homeo}(\T^2)$ the group of homeomorphisms on $\T^2$. 

\begin{theorem}\label{ABC3}

Suppose that $\Phi: \Gamma_B \to  \Diff^r(\T^2),\ r>2,$ is such that:
\begin{enumerate}
\item $B \in \SL_2(\Z)$ is an Anosov linear map (i.e. $B$ has its eigenvalues with norm different than one);
\item $\Phi(B)$ is an Anosov diffeomorphism of $\T^2$ homotopic to $A\in \SL_2(\Z)$.
\end{enumerate}

Then $\Phi$ is topologically conjugate to an affine action of $\Gamma_B$ up to finite index. More concretely, there exist a finite index subgroup $\Gamma' \subset \Gamma_B$ and  $h \in \mathrm{Homeo}(\T^2)$ such that $h\Phi(\gamma)h^{-1},$ for all $\gamma \in \Gamma'$, coincides with an affine action of the form \eqref{EqAffine}.

\end{theorem}
\begin{remark}
When $|tr A|\geq |tr B|$, the differentiability $r>2$ can be relaxed to $r\geq 1$. The assumption $r>2$ is only needed in the proof of the $|trA|<|trB|$ case where we apply the Herman-Yoccoz theory for circle maps. 
\end{remark}

This result gives a complete classification of  actions of ABC groups containing an Anosov element up to topological conjugacy and up to finite covers. It strengthens Theorem 1.5 of \cite{WX} by removing the slow oscillation condition therein.

\begin{remark} We do not know how to remove the assumption of the existence of an Anosov element in Theorem \ref{ABC3} since the proof relies strongly on the conjugacy linearizing an Anosov diffeomorphism provided by Franks' theorem, but it seems plausible the result still holds without this assumption in the smooth category. In the topological category, it is possible to construct faithful actions of $\Gamma_B$ by homeomorphisms in any surface (even by homeomorphisms homotopic to the identity) as the group $\Gamma_B$ is left-orderable if $B$ has infinite order and therefore the group $\Gamma_B$ admits an action on $\text{Homeo}(\R)$.

\end{remark}
\subsection{Smooth conjugacy}
We next show how to improve the regularity of the conjugacy $h$ in Theorem \ref{ABC3}. The main strategy is to use the ergodic effect of the $\Z^2$ part of the action to promote an SRB measure to a measure of maximal entropy.

In the next theorem, we get complete classification of faithful volume-preserving actions up to smooth conjugacy. 
\begin{theorem}\label{ThmHigher2}
Let $\Phi: \Gamma_B \to  \Diff^r(\T^2),\ r\geq 2,$ be a faithful action by volume-preserving diffeomorphisms satisfying:
\begin{enumerate}
\item $B \in \SL_2(\Z)$ is an Anosov linear map;
\item $\Phi(B)$ is an Anosov diffeomorphism.
\end{enumerate}
Then $h$ and $h^{-1}$ in Theorem \ref{ABC3} are $C^{r-\epsilon}$ for all $\epsilon>0$.
\end{theorem}

We remark that the volume-preservation assumption in the above Theorem \ref{ThmHigher2} might be unnecessary and we believe our methods might lead to a proof. Also, the volume-preservation assumption  can be replaced by the assumption of preserving a smooth measure.

In the proof of Theorem \ref{ThmHigher2} we can reduce to the case where $trA=trB$, since the rotation matrix $\rho$ has to be rational (hence is nonfaithful) if $tr A\neq tr B$, as we will see in the classification of affine actions given in Proposition \ref{LmNull} in Section \ref{SSAffine}. We will use the following notation.\\ 

\textbf{Notation:} For a diagonalizable matrix $D \in SL_2(\Z)$, we let $u_D$ to be a choice of an eigenvector associated to the larger in norm eigenvalue $\lambda_D$. We define $D^{-t}$ be the transpose of $D^{-1}$. For vectors $v,w \in \R^2$, we let $v \otimes w$ denote  the matrix product of $v$ treated as a column vector and $w$ treated as a row vector. More formally,  this matrix represents the endomorphism of $\R^2$ defined by the tensor product of the element $v$ of $\R^2$ and the element $w^{t}$ of the dual $(\R^2)^{*}$ (defined by the standard inner product).\\

We will see in Proposition \ref{LmNull} that a rotation matrix $\rho$ satisfying the equation $A\rho=\rho B+C$ for some $C\in \Z^{2\times 2}$ and $trA=trB$, $|trA|>2$, has the form
\begin{equation}\label{EqRotMat}\rho=\rho_C+c_1 u_A\otimes u_{B^t}+c_2 u_{A^{-1}}\otimes u_{B^{-t}},\quad (c_1,c_2)\in \R^2,\end{equation}
where $\rho_C\in \Q^{2\times 2}$ is a particular solution to $A\rho=\rho B+C$. 

An affine action of the form \eqref{EqAffine} is faithful if and only if $A$ is of infinite order and $\{\rho p,\ p\in \Z^2\}$ is dense on $\T^2$ (c.f. Proposition 1.2 of \cite{WX}). Therefore in terms of the representation \eqref{EqRotMat}, the rotation matrices of faithful affine actions correspond to a full measure set of $(c_1,c_2)\in \R^2$. In general the full measure set cannot be $\R^2$. For example, we consider $A=B$ then $\rho=\rho_C+ a_1 \mathrm{id}+a_2 A,\ (a_1,a_2)\in \R^2$ and the affine action is not faithful if we take $a_1,a_2\in \mathbb Q$. Given $\rho$, it can be checked directly if the affine action is faithful or not.

Next, replacing the above full measure set by another {\it implicit} full measure set, we obtain the following theorem without assuming the volume preservation assumption.




\begin{theorem}\label{ThmHigher3}
There exists a full measure subset $\mathcal C$ of $\R^2$ such that if  $\Phi: \Gamma_B \to  \Diff^r(\T^2),\ r\geq 2,$ is an action satisfying:
\begin{enumerate}
\item $B \in \SL_2(\Z)$ is an Anosov linear map;
\item $\Phi(B)$ is an Anosov diffeomorphism;
\item the rotation matrix \eqref{EqRotMat} has $(c_1,c_2)\in \mathcal C$.
\end{enumerate}
Then $h$ and $h^{-1}$ in Theorem \ref{ABC3} are $C^{r-\epsilon}$ for all $\epsilon>0$.
\end{theorem}

For special rotation vectors translating along the stable or unstable leaves of $A$, we can combine the Herman-Yoccoz theory for circle maps and the Ledrappier-Young entropy formula to obtain the following theorem.
\begin{theorem}\label{ThmHigher}
Suppose that $\Phi: \Gamma_B \to  \Diff^r(\T^2), \ r>2,$ is such that:
\begin{enumerate}
\item $B \in \SL_2(\Z)$ is an Anosov linear map;
\item $\Phi(B)$ is a volume-preserving Anosov diffeomorphism of $\T^2$ homotopic to $A\in \SL_2(\Z)$ with $trA=trB$;
\item The matrix $\rho=(\rho_1,\rho_2)$ in \eqref{EqAffine} has the form $\rho= c u_A\otimes u_{B^t}$ or $\rho= cu_{A^{-1}}\otimes u_{B^{-t}},\ c\in \R\setminus\{0\}.$\end{enumerate}
Then $h$ and $h^{-1}$ in Theorem \ref{ABC3} are $C^{r-\epsilon}$ for all $\epsilon>0$.
\end{theorem}

The proofs of Theorem \ref{ThmHigher2} and \ref{ThmHigher3} rely crucially on the two dimensionality. However, the proof of Theorem \ref{ThmHigher} admits a partial higher dimensional generalization as follows. 

Let $B \in \mathrm{SL}_n(\Z)$ $n\geq 2$ be an Anosov linear map with simple real spectrum. We extend the group $\Gamma_B$ to a group  $\Gamma_{B,n}= \Z \ltimes_B (\Z^n)^n$, where each $\Z^n$ factor form a copy of $\Gamma_B= \Z \ltimes_B \Z^n$ with $\Z$ and the $\Z^n$ factors commute. An affine action $\Phi $ of $\Gamma_{B,n}$ on torus $\T^n$ is given by $\Phi(B)=A\in \mathrm{SL}_n(\Z)$ and the $i$-th $\Z^n$ factor acts by translation with rotation matrix $\rho_i\in \R^{n\times n}$ such that $A\rho_i=\rho_i B$ mod $\Z^{n\times n}$ holds for all $i=1,\ldots,n$.

\begin{theorem}\label{ThmHigher4}
Let $\Phi: \Gamma_{B,n} \to  \Diff^r(\T^n),\ r> n,$ be an action satisfying:
\begin{enumerate}
\item $B \in \SL_n(\Z)$ is an Anosov linear map with simple real spectrum $\lambda_1,\ldots,\lambda_n$, and $A\in \SL_n(\Z)$ has the same spectrum;
\item $\Phi(B)$ is volume-preserving Anosov $C^1$ close to $A$;
\item There is a homeomorphism $h$ conjugating the action to an affine action satisfying $\rho_i=c_i u_{A,i}\otimes u_{B^t,i},$ where $c_i\neq 0$ and  $u_{A,i}$ (resp. $u_{B^t,i}$) is the normalized eigenvector of $A$ (resp. $B^t$) corresponding to eigenvalue $\lambda_i,  i=1,\ldots,n.$  Moreover, the Anosov part of the affine action is $A$. 
\end{enumerate}
Then $h$ and $h^{-1}$   are $C^{1,\nu}$ for some $\nu>0$.
\end{theorem}

We have the following remarks improving the result slightly.
\begin{remark}
\begin{enumerate}
\item The $c_i$ corresponding to one of the largest in norm or smallest in norm eigenvalue is allowed to be zero by the volume preserving condition.
\item In the case of $n=3$, when $\Phi(B)$ is sufficiently smooth ($C^r$ for $r$ large enough), the conjugacy $h$ can be $C^{r-\eps}$ for all $\eps>0$ applying \cite{Go2}.
\end{enumerate}
\end{remark}

In recent years, there has been an emerging interest in understanding actions of non-abelian groups on manifolds by homeomorphisms or diffeomorphisms. Among the simplest interesting examples are solvable groups. For example, for the Bamslag-Solitar group $$\mathrm{BS}(p,q)=\langle a,b\ |\ ab^pa^{-1}=b^q\rangle,$$ there are results by  Farb and Franks \cite{FF}, by Burslem and Wilkinson \cite{BW},  who classified analytic $\mathrm{BS}(1,n)$ action on $\mathbb S^1$, which was later improved in \cite{GL1,BMNR}. And there are results about  Baumslag solitar group actions on surfaces by \cite{ARX,AGRX} and \cite{GL2,GL3}. 

Actions of ABC groups have been studied by many authors.  In \cite{A1,A2}, Asaoka studied the local rigidity of some non polycyclic $(|\det B|>1)$ ABC group actions on $\T^n$ and $\mathbb S^n$ generalizing \cite{BW}.  In \cite{WX}, Wilkinson and the second author studied the local rigidity for polycyclic $(|\det B|=1)$ ABC group actions on $\T^n$ using a KAM method and also obtained a number of semi-global rigidity results by combining the Herman-Yoccoz theory for circle maps and the geometry of the invariant foliations. In \cite{MC}, McCarthy proved certain $C^1$ local rigidity of the trivial action of ABC groups on compact manifolds, which is generalized to actions by semi products in \cite{BKKT} by Bonatti-Kim-Koberda-Triestino. In \cite{BMNR}, Bonatti-Monteverde-Navas-Rivas obtained a number of rigidity results for the $C^1$ action of ABC groups on interval with some hyperbolicity. We also note that the local rigidity result of \cite{W} was proved  in the Lie group setting which can be considered as a continuous analogue of our discrete ABC group action. In \cite{L}, it is proved a local rigidity result for a solvable group ($\Z\ltimes_\lambda\R$) acting on $\T^n$. 


\begin{subsection}{Free subgroups of diffeomorphism groups}
One novel discovery in the proof of Theorem \ref{ABC3} is a version of Tits alternative for subgroups of the diffeomorphism group $\Diff^1(\T^2)$ containing an Anosov element. 

The Tits alternative is a well-known theorem stating that a linear group (i.e. a subgroup of the group of invertible matrices $\mathrm{GL}_n(\R)$) either contains a free subgroup on two generators  or it is virtually solvable. This result has far-reaching applications in various fields in mathematics such as group theory, geometry and dynamics, etc.  Non-linear transformation groups such as the groups of homeomorphisms or diffeomorphisms of a closed connected manifold $M$ (denoted by $\text{Homeo}(M)$ and $\Diff(M)$ respectively) are conjectured to behave in many ways similar to linear groups (See \cite{Fis}) and therefore one can ask whether the Tits Alternative still holds for such groups.

Unfortunately, the Tits alternative does not hold for these groups, since $\text{Homeo}(\s^1)$ and $\Diff(\s^1)$ contain subgroups isomorphic to Thompson's group $F$, which is known to not contain free subgroups and neither be solvable (c.f. \cite{N,GS}). Nonetheless, up to what extent such an alternative fails for $\text{Homeo}(M)$ and $\Diff(M)$ is not clear.  As an example, a well known theorem of Margulis (Conjectured by Ghys) for the group of homeomorphisms of the circle $\s^1$ states the following:

\begin{theorem}\label{margulistits}(Margulis)
A group $\Gamma \subseteq \mathrm{Homeo}(\s^1)$ either contains a free subgroup on two generators or preserves a probability measure on $\s^1$.
\end{theorem}

In the case when the dimension of $M$ is greater than one, very few similar results are known due to the complexity of elements in $\mathrm{Homeo}(M)$, but there have been some recent results in this direction. For example, Franks-Handel \cite{FH} have shown that groups of real-analytic area-preserving diffeomorphisms of $\s^2$ which contain certain entropy zero diffeomorphisms satisfy the Tits alternative.

Our discovery is a Tits alternative for $\Diff^1(M)$ in the presence of an Anosov element.  To fix the notations,  for an Anosov diffeomorphism $f$ on $M$, we denote  by $W^s_f$ and $W^u_f$ its stable and unstable foliations  respectively. 

\begin{theorem}\label{main2}
Let  $\Gamma$  be a subgroup of $\Diff^1(\T^2)$ containing an Anosov element $f \in \Diff^1(\T^2)$, then there is a dichotomy:

\begin{enumerate}
\item either there is an index 2 subgroup of $\Gamma$ which preserves either $W_f^s$ or $W_f^u$;
\item or $\Gamma$ contains a free subgroup $\mathbb F_2$.
\end{enumerate}

\end{theorem}

The above result relies crucially on the two dimensionality. We next state two results as partial generalizations of Theorem \ref{main2} to higher dimensions. We denote further by $E^s_f$ and $E^u_f$ the stable and unstable distributions of an Anosov diffeomorphism $f$ respectively. The next result is the main technical ingredient in the proof of Theorem \ref{main2}.

\begin{theorem}\label{main3} Let $\Gamma$ be a subgroup of  $\diff^1(M)$ satisfying the following:
\begin{enumerate}
\item There exists  $f \in \Gamma$ which is a transitive Anosov diffeomorphism such that $\dim W_f^s = \dim W_f^u$.
\item There exists $h \in \Gamma$ and $x \in M$ such that each of the subspaces $D_xh(E^s_{f}(x))$ and $D_xh(E^u_{f}(x))$ intersect transversely both $E^s_{f}(h(x))$ and $E^u_{f}(h(x))$.
\end{enumerate}
Then, there exists an integer $N$ such that $\{f^N, hf^{N}h^{-1}\}$ generate a free group $\mathbb{F}_2$.
\end{theorem}
Relaxing assumption (2) in Theorem \ref{main3} we have the following result.
\begin{theorem}\label{MainSemi}
Let $\Gamma$ be a subgroup in $\Diff^1(M)$ satisfying
\begin{enumerate}
\item there is a transitive Anosov element $f\in \Gamma$ with stable and unstable foliations $W_f^s$ and $W_f^u$ respectively and $\mathrm{dim}W_f^s\leq \mathrm{dim}W_f^u$;
\item there exist an element $h\in \Gamma$ and a point $x\in M$, such that $D_xh E_f^u(x)$ intersects transversally $E_f^u(h(x))$.
\end{enumerate}
Then there exists an integer $N$ such that $\{f^N, hf^N h^{-1}\}$ generate a nontrivial free semi-group.
\end{theorem}



\end{subsection}

The paper is organized as follows:
In Section \ref{SPingPong}, we first give an outline and then proceed to prove Theorem \ref{main3} and Theorem \ref{MainSemi}. The main observation in the proof is a ping-pong argument.
 In Section \ref{S2D}, we prove Theorem \ref{main2}.
 The next four sections are devoted to the proof of Theorem \ref{ABC3}. The proof is divided into three cases depending on the comparison of $|\lambda_A|$ and $|\lambda_B|$ respectively. In Section \ref{SABC1} we prove Theorem \ref{ABC3} in the case of $|\lambda_A|> |\lambda_B|$. In Section \ref{mistakecorrection}, we prove Theorem \ref{ABC3} in the case of $|\lambda_A|= |\lambda_B|$. In Section \ref{SABC2} we prove Theorem \ref{ABC3} in the case of $|\lambda_A|< |\lambda_B|$.
 In Section \ref{SHigher}, we prove Theorem \ref{ThmHigher} and \ref{ThmHigher4}.
  In Section \ref{SHigher2}, we prove Theorem \ref{ThmHigher2}.
 In Section \ref{SHigher3}, we prove Theorem \ref{ThmHigher3}.
\end{section}

\begin{section}{Ping-Pong: Proofs of Theorem \ref{main3} and Theorem \ref{MainSemi}}\label{SPingPong}
In this section, we prove Theorem \ref{main3} and Theorem \ref{MainSemi}. We first give an outline of the proof in Section \ref{SSOutline}. Afterwards, we give the formal proof. We will start by recalling some notations and results from hyperbolic dynamics in Section \ref{SSHyp}. In Section \ref{SSPP}, we give the ping-pong argument. In the remaining two subsections, we prove Theorem \ref{main3} and Theorem \ref{MainSemi}.
\subsection{Outline of the proof of Theorem \ref{main3}}\label{SSOutline}
We now proceed to give an outline of the proof of Theorem \ref{main3}.

If we consider the two Anosov diffeomorphism $f_1 := f$ and $f_2:= hfh^{-1}$, both the stable foliations and unstable foliations of $f_1$ and $f_2$ are transverse in some open neighborhood of by assumption 2) in Theorem \ref{main3}. The existence of a free group follows from a ping-pong argument using the following elementary observation: Suppose $B$ is a small open ball in $M$, if $N>0$ is large enough, the ball $f_1^N(B)$ looks similar to a small neighborhood of a long piece of a leaf of the unstable foliation $W_{f_1}^u$, as the foliation $W_{f_1}^u$ is transverse to both foliations $W_{f_2}^s$ and $W_{f_2}^u$, by the inclination lemma \ref{inclination} the set $f_2^Nf_1^N(B)$ (or $f_2^{-N}f_1^N(B)$) must contain an open set that looks like an open neighborhood of a leaf of $W_{f_2}^u$ (or $W_{f_2}^s$) near every point in $M$ and so we conclude that the words $f_2^Nf_1^N$ and $f_2^{-N}f^N_1$ are non-trivial elements in $\Gamma$, see Figure \ref{pain2}. In general, we will show that if $N$ is large enough, for any non-trivial word $w$ in the elements $f_1^N$ and $f_2^N$ the set $w(B)$ will contain a subset that looks like  the stable or unstable foliations of $f_1$ or $f_2$ and this will guarantee that $w(B) \neq B$ and so the group generated by $\{f_1^N, f_2^N\}$ must be necessarily free.\\

\subsection{Preparations from hyperbolic dynamics}\label{SSHyp}
We recall some definitions and well known facts about hyperbolic dynamics and prepare notations for later use. Let $p$ be a hyperbolic fixed point of a diffeomorphism $f$ of $M^m$.
\begin{definition}
\begin{enumerate}
\item By Theorem 6.2.3 in \cite{KatHas}, there exists a neighborhood $O$ of $p$ and a coordinate chart $\phi: O \to \R^m$ such that $\phi(W^u_f(p) \cap O) \subset \R^k \oplus \{0\}$ and $\phi(W^s_f(p) \cap O) \subset \{0\} \oplus  \R^{m-k}$ where $m = \dim(M)$ and $k = \dim(W_f^u(p))$,  this construction is known as \emph{ ``adapted coordinates"} for the dynamics of $f$ near $p$.
\item In those coordinates there are obvious \emph{projections} from $\R^m$ into $\R^k$ or $\R^{m-k}$  that we denote by $\pi_1, \pi_2$ respectively.
\item In the decomposition $\R^k\oplus \R^{m-k}$, we say that $\R^k$ is \emph{ horizontal} and $\R^{m-k}$ is \emph{ vertical.}
\end{enumerate}
 \end{definition}
\begin{definition}
\begin{enumerate}
\item Let $\mathbb{D}_k \subset \R^k$ be the $k$-dimensional open unit disk,  a \emph{ $C^1$ embedded $k$-disk $D$} in  $M$ is defined as a $C^1$ embedding $f: \mathbb{D}_k \to M$.
\item Consider the  $k$-Grassmanian $\mathrm{Gr}_k(M)$ of $M$, which is a fiber bundle over $M$ whose fiber over point $p$ of $M$ consist of all the $k-$dimensional subspaces of the tangent space $T_p(M)$, we choose an arbitrary metric on $\mathrm{Gr}_k(M)$ that we denote by $d$. A $C^1$ map  $f: \mathbb{D}_k \to M$ induces a natural map $G_k(f): \mathbb{D}_k \to \mathrm{Gr}_k(M)$ via $G_k(f)(p)=(f(p), D_pf(T_p\mathbb{D}_k) )$.
\item We say that two disks $f_1: \mathbb{D}_k \to M$ and $f_2: \mathbb{D}_k \to M$ are \emph{ $C^1$  $\e$-close} if for every $p \in \mathbb{D}_k$ we have  $d(G_k(f_1(p)), G_k(f_2(p))) < \e$.
\end{enumerate}
\end{definition}

Our main technical tool is the following theorem known as the Inclination Lemma or $\lambda$-Lemma.

\begin{lemma}[Inclination Lemma, Lemma 6.2.23 \cite{KatHas}]\label{inclination}
Considering $C^1$ adapted coordinates on a neighborhood $O$ of a hyperbolic fixed point $p$ of a diffeomorphism $f: M \to M$. Given $\e, K, \nu > 0$, there exists $N$ such that if $D$ is a $C^1$ embedded $k$-disk containing $q \in W_f^{s} \cap O$ with all tangent spaces in horizontal $K$-cones and such that $\pi_2(D)$ contains a $\nu$-ball around $0 \in \R^{k}$, then there exists $D_1 \subset D$ embedded disk such that for any $n \geq N$, the set $\pi_1(f^n(D_1))$ contains a unit ball in $\R^k$ and $T_zf^n(D_1)$ is contained in a horizontal $\e$-cone for every $z \in f^n(D_1)$.
\end{lemma}

\begin{remark} See (\cite{KatHas}, Sec. 6.2) for the definition of cones. In \cite{KatHas} the theorem is stated slightly different but the equivalence is easy. See also Lemma 5.7.1. in \cite{BriStu}.
\end{remark}

We have the following  corollary of the inclination lemma.
\begin{corollary}\label{cheating} Suppose that $D_1$ is a $C^1$ embedded $k$-disk intersecting $W_f^s(p)$ transversely and $D_2 \subset W_f^u(p)$ is an embedded $k$-disk, then there exists $\e_0>0$ (just depending on $D_1$) such for any $0<\e < \e_0$ there exists $N>0$ such that if $D$ is a $k$-disk which is $C^1$ $\e_0$-close to $D_1$, then $D$ intersects transversely $W_f^s(p)$, and for all $n \geq N$, $f^n(D)$ contains a $C^1$ embedded disk $D'$ which is $C^1$ $\e$-close to $D_2$.
\end{corollary}

For a fixed Riemannian metric on $M$, $x \in M, \delta >0$, let $W^{s}_f(p,\delta)$ and $W^{u}_f(p,\delta)$ be the local stable and unstable manifolds at distance $\delta$, which are the subsets of $M$ consists of points whose forward orbit (or backward orbit) stay at distance less than $\delta$ to the orbit of $p$.  If $p$ is a hyperbolic periodic point and $\delta$ is small enough, then $W^{s}_f(p,\delta)$ and $W^u_f(p, \delta)$ are embedded manifolds.

\subsection{The Ping-Pong argument}\label{SSPP}
The main idea in the proof of Theorem \ref{main3} is enclosed in the following proposition:

\begin{proposition}[Ping Pong argument]\label{pingpong} Suppose $f_1,f_2 \in \Diff(M)$ have hyperbolic fixed points $p_1,p_2 \in M$ and $\delta>0$ is such that the local stable and unstable disks $W^{s}_{f_1}(p_1,\delta)$, $W^{u}_{f_1}(p_1,\delta)$ are smooth and both disks intersect transversely and non-trivially both $W^{s}_{f_2}(p_2,\delta)$ and $W^{u}_{f_2}(p_2,\delta)$, then there exists $N>0$ such that for all $n \geq N$ the group generated by $\{f_1^n, f_2^n\}$ is free.
\end{proposition}

\begin{proof}

Let $O_1$, $O_2$ be adapted coordinates for $f_1$ and $f_2$ with corresponding system of coordinates $\phi_i: O_i \to \R^n$, we can assume that $O_1$ and $O_2$ are disjoint sets  and for simplicity assume that $\phi_i(O_i)$ contains the ball $B^n_4(0) \subset \R^n$ of radius 4 for each $i=1,2$. Let $$C_i := \phi_i^{-1} \bigg( \{2\} \times B^k_{1}(0) \cup B^{n-k}_{1}(0)\times \{2\} \bigg)$$ and let $C = C_1 \cup C_2$. Therefore $C$ is the union of four ``small'' disks, each one parallel to one of the stable or unstable manifolds of $f_1$ at $p_1$ or $f_2$ at $p_2$. See figure \ref{pain2}.\\

\begin{figure}
\begin{minipage}[t]{0.5\linewidth}
\centering
\includegraphics[width=2.2in]{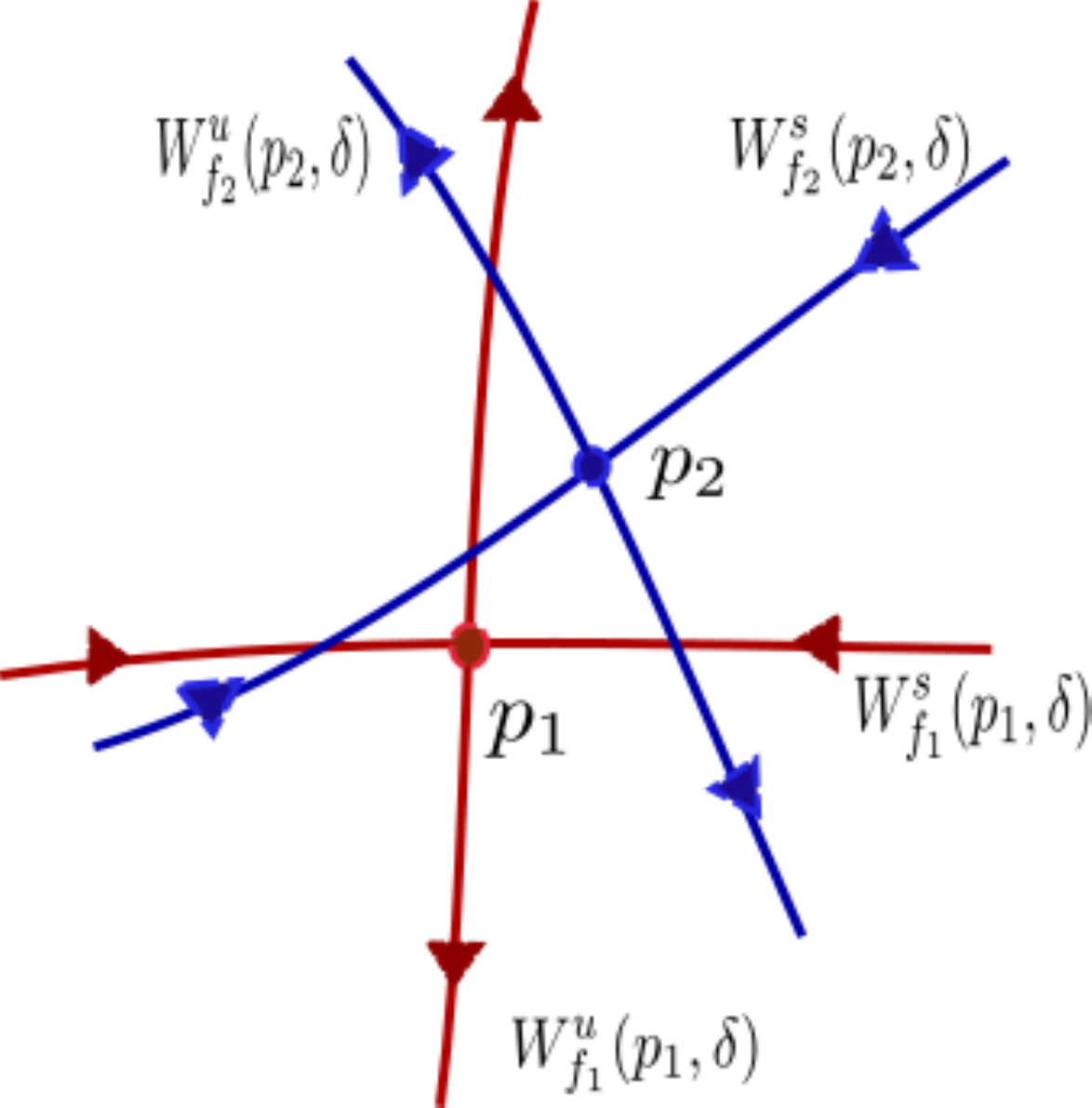}
\end{minipage}%
\begin{minipage}[t]{0.5\linewidth}
\centering
\includegraphics[width=2.5in]{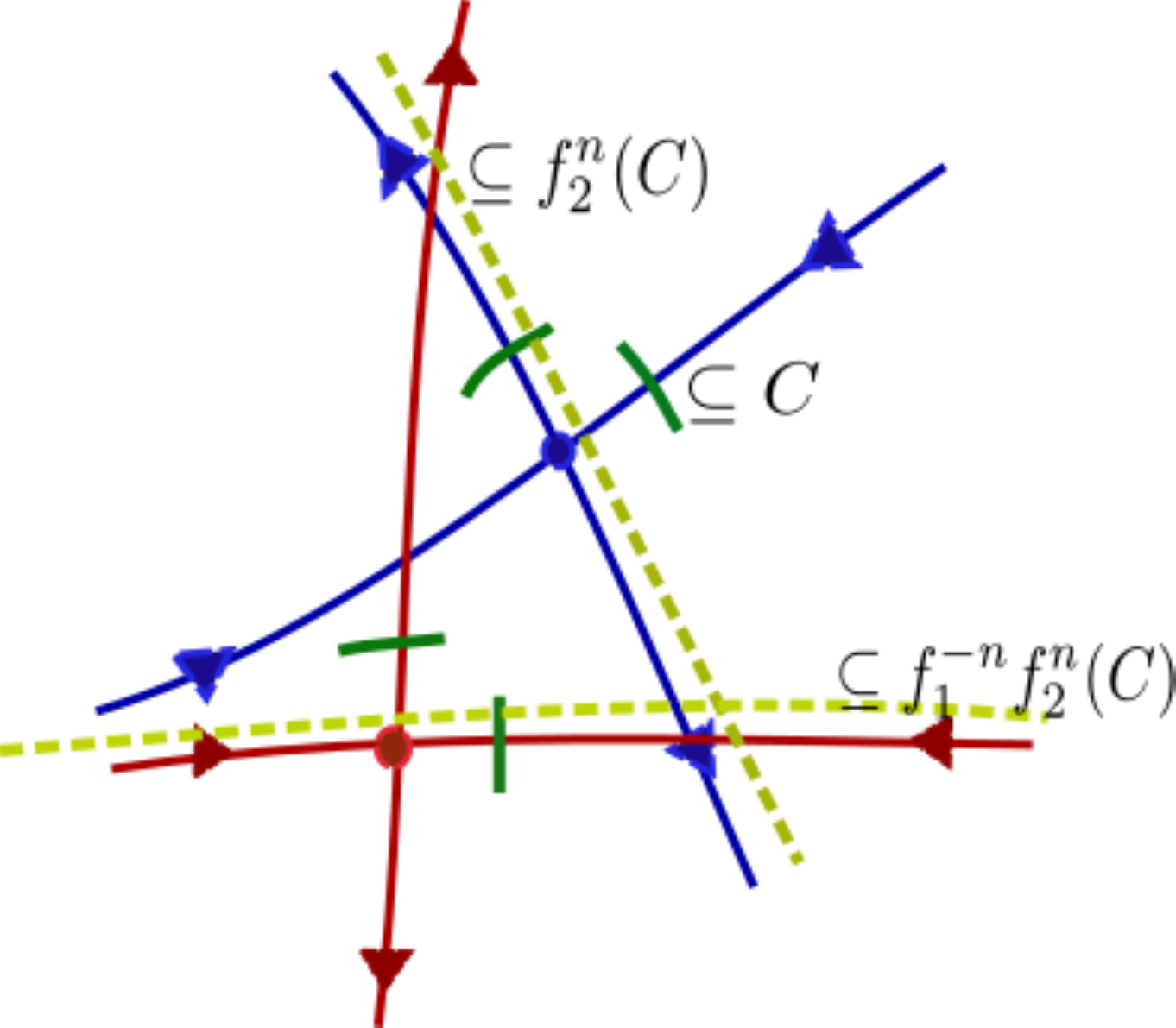}
\end{minipage}
\caption{In the right, the set $C$ consists of the four small segments colored in green, the yellow dashed lines correspond to disks contained in $f_2^n(C)$ and $f_1^{-n}f_2^n(C)$.}
\label{pain2}
\end{figure}


As an application of Corollary \ref{cheating}, there exists $\e_1>0$, such that for any $0 < \e < \e_1$, there is $N_1>0$ such that if a $C^1$ smooth disk $D$ is $\e_1$-close to $W_{f_1}^s(p_1, \delta)$ (or $W_{f_1}^u(p_1, \delta)$), then $D$ intersects transversely  $W_{f_2}^s(p_2, \delta)$ and $W_{f_2}^u(p_2, \delta)$ and if $n \geq N_1$, then $f_2^n(D)$ contains a disk which is $C^1$ $\e$-close to  $W_{f_2}^u(p_2, \delta)$ and $f_2^{-n}(D)$ contains a disk which is $C^1$ $\e$-close to $W_{f_2}^s(p_2, \delta)$.

The same argument implies there exists $\e_2>0$, such that for any $0 < \e < \e_2$ there is $N_2>0$ so that if a $C^1$ smooth disk $D$ is $\e_2$-close to  $W_{f_2}^s(p_2, \delta)$ (or $W_{f_2}^u(p_2, \delta)$), then $D$ intersects transversely both $W_{f_1}^s(p_1, \delta)$ and $W_{f_1}^u(p_1, \delta)$ and if $n \geq N_2$, then$f_1^n(D)$ contains a disk which is $C^1$ $\e$-close to  $W_{f_1}^u(p_1, \delta)$ and  also$f_1^{-n}(D)$contains a disk which is $C^1$ $\e$-close to $W_{f_1}^s(p_1, \delta)$.

Also, from the inclination Lemma, for any $\e>0$ there exists $N_3>0$ so that if $n\geq N_3$, then for every $i= 1,2$,  $f_i^n(C)$ contains an embedded disk $\e$-close to $W^u_{f_i}(p, \delta)$ and that  $f_{i}^{-n}(C)$ contains an embedded disk $\e$-close to $W^s_{f_i}(p, \delta)$.

Let $\e = \min \{\e_1,\e_2 \}$ and let $N = \max \{N_1,N_2, N_3\}$. By definition of $\e_1,\e_2, N_1, N_2, N_3$ one obtains by an easy induction in the word length the following:

\begin{claim}\label{realwork} If $n \geq N$, for any non trivial word $w$ in  $\{f_1^n, f_2^n, f_1^{-n}, f_2^{-n}\}$ we have:

\begin{enumerate}

\item If $w$ starts by $f_1^n$ ($f_2^n$ respectively),  the set $w(C)$ contains an embedded disk $D$ that is $C^1$ $\epsilon$-close to $W^{u}_{f_1}(p_1,\e)$  $(W^{u}_{f_2}(p_2,\e)$ respectively$)$.
\item if $w$ starts by $f_1^{-n}$ ($f_2^{-n}$ respectively),  the set $w(C)$ contains an embedded disk $D$ that is $C^1$ $\epsilon$-close to $W^{s}_{f_1}(p_1,\e)$  $(W^{s}_{f_2}(p_2,\e)$ respectively$)$.

\end{enumerate}

\end{claim}

As a consequence of Claim \ref{realwork}, for any non-trivial word $w(C)$ contains an embedded disk $C^1$ $\e$-close to one among  $W^u_{f_i}(p_i)$ or $W^s_{f_i}(p_i)$ for some $i$ and in particular choosing $\e$ small enough, we can guarantee that $w(C) \neq C$. In that case, the group generated by $\{f_1^n, f_2^n\}$ is free.

\end{proof}
\subsection{Proof of Theorem \ref{main3}}
We now finish the proof of Theorem \ref{main3}:

\begin{proof}[Proof of Theorem \ref{main3}.]As the condition of $D_xh(E^s_{f}(x))$ and $D_xh(E^u_{f}(x))$ intersecting transversely both $E^s_{f}(h(x))$ and $E^u_{f}(h(x))$ is an open condition on $x \in M$ and periodic points of $f$ are dense (consequence of the Anosov Closing lemma and $f$ being transitive) we can assume that $x$ is a periodic point of $f$.

By choosing $\delta$ small enough we can assume that $W^{s}_f(x,\delta)$ and $W^{u}_f(x,\delta)$ are embedded disjoint disks such that each of the disks $h(W^{s}_f(x,\delta))$ and $h(W^{u}_f(x,\delta))$ are both transverse to $W^s_f(h(x))$ and $W^u_f(h(x))$.

Let $p_2:= h(x)$, as the periodic points of $f$ are dense, we can assume that  there exists a  periodic point $p_1$ of $f$ near the point $p_2$ such that the local stable and unstable manifolds $W^{s}_{f}(p_1, \delta)$ and $W^{u}_{f}(p_1, \delta)$ are embedded manifolds that intersect transversely $h(W^{s}_{f}(x, \delta))$ and $h(W^{u}_{f}(x, \delta))$ (we are using that $W^s$ and $W^u$ vary continuously). By taking $p_1$ close enough to $h(x)$, we can assume that the intersection of any of the disks $h(W^{s}_{f}(x, \delta))$ and $h(W^{u}_{f}(x, \delta))$ with any of the disks $W^{s}_{f}(p_1, \delta)$ (or $W^{u}_{f}(p_1, \delta)$) is a transverse intersection at exactly one point. See figure \ref{pain2}.

Let $f_1:= f^{N}$ and $f_2:= hf^{N}h^{-1}$ for some $N>0$. Observe that both $f_1$ and $f_2$ are Anosov diffeomorphisms, by choosing $N$ multiple of both the periods of $p_1$ and $p_2$, we can assume that the points $p_1$, $p_2$  are hyperbolic fixed points of $f_1^{N}$ and $f_2^{N}$. We can therefore apply proposition \ref{pingpong} and we are done.

\end{proof}
\subsection{Proof of Theorem \ref{MainSemi}}
Similarly, we finish the proof of Theorem \ref{MainSemi}.
\begin{proof}[Proof of Theorem \ref{MainSemi}]
Similar to the proof of Theorem \ref{main3}, by the assumption, we can find $N_1$ and $p_1,p_2$ such that the following hold:

\begin{enumerate}
\item $f_1=f^{N_1},\ f_2=h f^{N_1} h^{-1}$ have $p_1$ and $p_2$ as hyperbolic fixed points respectively;
\item $W^u_{f_1}(p_1)$ intersects $W^u_{f_2}(p_2)$ transversally.
\end{enumerate}

We can now repeat the proof of Proposition \ref{pingpong} for words generated by $\{f^{N_2}_1,f^{N_2}_2\}$ for some $N_2$ large to show that $\{f^{N_2}_1,f^{N_2}_2\}$ generates a semigroup. Indeed, we pick a disk $D_i$ of dimension dim$W^u_f$ near $p_i$ transverse to $W^s_{f_i}(p_i)$ and disjoint from  $W^u_{f_i}(p_i)$, $i=1,2$.  After the application of any word starting with $f^{N_2}_1$, the image of $D_1$ would be  very close to either of $W^u_{f_i}(p_i)$, so is transverse to both $W^u_{f_{2-i}}(p_{2-i})$ and $W^u_{f_{2-i}}(p_{2-i})$, $i=1,2$. In particular, the image can never by $D_1$.
 Similarly for a word starting with $f^{N_2}_2$, we look at the image of $D_2$ instead. This completes the proof.
\end{proof}
\end{section}

\begin{section}{Proof of Theorem \ref{main2}}\label{S2D}

In this section we give a proof of Theorem \ref{main2}. Observe that by Theorem \ref{main3}  to obtain a free subgroup of $\Gamma$, we only need to show the existence of an element $h \in \Gamma$ and a point $x \in \T^2$ such that the subspaces $D_xh(E^s_f(x))$ and $D_xh(E^u_f(x))$ intersect transversely both $E^s_f(h(x))$ and $E^u_f(h(x))$. This will be shown in Proposition \ref{trans}.

\begin{proposition}\label{hardtrans} Suppose $f$ is an Anosov diffeomorphism in a subgroup $\Gamma \subset \Diff^1(\T^2)$ and suppose there exist $h_1, h_2 \in \Gamma$ and $x_1, x_2 \in \T^2$ such that:

\begin{enumerate}
\item $D_{x_1}h_1(E^s_{f}(x_1))$ intersect transversely both $E^s_{f}(h_1(x_1))$ and $E^u_{f}(h_1(x_1))$.
\item $D_{x_2}h_2(E^u_{f}(x_2))$ intersect transversely both $E^s_{f}(h_2(x_2))$ and $E^u_{f}(h_2(x_2))$.
\end{enumerate}

Then, there exists $h \in \Gamma$ and $x \in \T^2$ such $D_xh(E^s_{f}(x))$ and $D_xh(E^u_{f}(x))$ intersect transversely both $E^s_{f}(h(x))$ and $E^u_{f}(h(x))$.

\end{proposition}

\begin{proof}
If $D_{x_1}h_1(E^u_{f}(x_1))$ intersect transversely both $E^s_{f}(h_1(x_1))$ and $E^u_{f}(h_1(x_1))$ then we can take $h= h_1$ and $x = x_1$ and we are done. Otherwise we have either of the following two cases holds:
\begin{enumerate}
\item[(a)]  $D_{x_1}h_1(E^u_{f}(x_1))=E^s_{f}(h_1(x_1))$;
\item[(b)] $D_{x_1}h_1(E^u_{f}(x_1))=E^u_{f}(h_1(x_1))$.
  \end{enumerate}
   Suppose the former is true, i.e. we have  $D_{x_1}h_1(E^u_{f}(x_1))=E^s_{f}(h_1(x_1))$.

We may further assume that $D_{x'}h_1(E^u_{f}(x'))=E^s_{f}(h_1(x'))$ for all $x'$ sufficiently close to $x_1$. The assumption (1) still holds with $x'$ in place of $x_1$ by the continuity.  We also have $D_{x'}h_1(E^u_{f}(x'))\neq E^u_{f}(h_1(x'))$ since we have assumed $ D_{x_1}h_1(E^u_{f}(x_1))\neq E^u_{f}(h_1(x_1))$.  If $D_{x'}h_1(E^u_{f}(x'))\neq E^s_{f}(h_1(x'))$ for some point $x'$ close to $x_1$, we then complete the proof by choosing $h=h_1$ and $x=x'$.

By the transitivity of $f$, we choose $x'$ sufficiently close to $x_1$ such that
\begin{itemize}
\item $h_1(x')$ lies in a dense backward orbit of $f$;
\item (1) holds with $x'$ in place of $x_1$;
\item $D_{x'}h_1(E^u_{f}(x'))=E^s_{f}(h_1(x'))$.
\end{itemize}

For any $N$ we have $D_{x'}f^{-N}h_1(E^u_{f}(x'))=E^s_f(f^{-N}h_1(x'))$. Next,
by assumption (1) and taking $N$ sufficiently large we have that $D_{x'}f^{-N}h_1(E^s_{f}(x'))$ is as close as we want to $E^s_f(f^{-N}h_1(x'))$.

 By the first bullet point and the assumption (1), we can choose $N$ such that $f^{-N}h_1(x')$ and $x_1$ are close enough to guarantee that the image under $Dh_1$ of both $D_{x'}f^{-N}h_1(E^u_{f}(x'))$ and $D_{x'}f^{-N}h_1(E^s_{f}(x'))$ are transverse to $E^s(h_1f^{-N}h_1x')$ and $E^u(h_1f^{-N}h_1x')$. The proposition is then proved by choosing $h=h_1f^{-N}h_1$ and $x=x'$.

 The argument for the case (b) is similar. We have $D_{x'}f^{N}h_1(E^u_{f}(x'))=E^u_f(f^{N}h_1(x'))$ for all $N$, and by (1) in the assumption $D_{x'}f^{N}h_1(E^s_{f}(x'))$ is arbitrarily close to $E^u_f(f^{N}h_1(x'))$. Hence we complete the proof by taking $h=h_2 f^N h_1$ and apply (2) in the assumption.
\end{proof}

The following proposition finishes the proof of Theorem \ref{main2}.

\begin{proposition}\label{trans} Suppose $f \in \Gamma$ is an Anosov diffeomorphism of $\T^2$ and let $W^s$ and $W^u$ be the stable and unstable foliations of $f$. One of the following holds:

\begin{enumerate}
\item There exists an index 2-subgroup $\Gamma_0 \subset \Gamma$ preserving one among $W^s$ or $W^u$.
\item There exists $h \in \Gamma$ and $x \in \T^2$ such that each of $D_xh(E^s_f(x))$ and $D_xh(E^u_f(x))$ intersect transversely both $E^s_f(h(x))$ and $E^u_f(h(x))$.
\end{enumerate}

\end{proposition}

\begin{proof}
If Condition (2) does not hold, then one of the conditions in Proposition \ref{hardtrans} does not hold, suppose the first condition does not hold, that is, suppose that for every $h \in \Gamma$ and every $x \in M$ we have that  $D_xh(E^s_f(x))$  does not intersect transversely one among $E^s_f(h(x))$ and $E^u_f(h(x))$ ,  as $\T^2$ is 2-dimensional then we must have  $D_xh(E^s_f(x))$ coincides with  $E^s_f(h(x))$ or $E^u_f(h(x))$ and as $\T^2$ is connected, this imply that $h(W^s)= W^s$ or $h(W^s) = W^u$, and this implies that the subgroup $$\Gamma_0 = \{h \in \Gamma \text{ such that } h(W^s)= W^s\}$$ is a subgroup of index 2 in $\Gamma$ preserving $W^s$.\\
\end{proof}

\end{section}

\begin{section}{Global rigidity of Anosov ABC actions on $\T^2$: Proof outline}\label{SABC0}
\subsection{Outline of the global rigidity}
In this section we give an outline of the proof of Theorem \ref{ABC3}. We will proceed to carry out this outline in the next sections. Recall that for a matrix $B \in \SL_2(\Z)$, we let $\Gamma_B := \Z \ltimes_{B} \Z^2$ be the corresponding Polycyclic Abelian-by-cyclic (ABC) group, which is a semi-direct product where $\Z$ acts in $\Z^2$ by $B$ and for convenience we denote the generator of $\Z$ in $\Gamma_{B}$ by $B$ as well.

Recall that we are concerned with understanding homomorphisms $\Phi: \Gamma_B \to \Diff(\T^2)$ such that $\Phi(B)$ is an Anosov diffeomorphism. By Franks Theorem \cite{F} there exists a homeomorphism $h \in \text{Homeo}(\T^2)$ such that the conjugate $A:= h\Phi(B)h^{-1}$ is an Anosov linear map, we denote by $\lambda_A$ and $\lambda_B$ the larger in modulus eigenvalue of $A$ and $B$ respectively.

The proof of Theorem \ref{ABC3} has several steps:

\textbf{Step 1}: By using our ping pong result (Theorem \ref{main2}) and using the fact that $\Gamma_B$ does not contain any free subgroup we can reduce to the case where $\Phi(\Gamma_B)$ preserves the unstable (or stable) foliation of $\Phi(B)$.

\textbf{Step 2}: We must show that for all $v \in \Z^2 \subset \Gamma_B$,  $h\Phi(v)h^{-1}$ of $\text{Homeo}(\T^2)$ is a translation. We first use the fact that $\text{PSL}_2(\Z)$ is isomorphic to $\Z_2*\Z_3$ to show that each $\Phi(v)$ is homotopic to identity (Corollary \ref{ABCeasy}). Next, we show that for any $v \in \Z^2$, the Misiurewicz-Ziemian rotation set of $\Phi(v)$ consists of a single point (Proposition \ref{difficult}). This is carried out as follows. Step 1 allow us to show that for $v \in \Z^2$ the rotation set of a lift of $\Phi(v)$ consists of either a single rotation vector or is a non-trivial segment parallel to the unstable eigenvector of $A$ and so we have to rule out this later possibility. While understanding the rotation set of elements of $\Phi(\Z^2)$, we were lead into understanding the joint rotation set (See Definition \ref{jointrota}) which is the set of possible  average rotation vectors for $\Phi(\Z^2)$-invariant measures, this set has a natural invariance by $A$ and $B$, and by using the hyperbolicity of $A$ and $B$ and basic facts about the joint rotation set, we will conclude that the actual rotation set of $\Phi(v)$ is a singleton for all $v \in \Z^2 \subset \Gamma_B$. The proof that the rotation set is a singleton is much easier in the case where $|\lambda_A| \neq |\lambda_B|$, the case where $|\lambda_A| = |\lambda_B|$ require another argument which is given in Section \ref{mistakecorrection}.

After showing that rotation sets consists of singletons, the proof is divided into two cases:

\textbf{Step 3}: The case where $|\lambda_A| \geq |\lambda_B|$: In the case where $|\lambda_A| > |\lambda_B|$ , using that rotation set consists of singletons and conjugating by high powers of $B$ one shows that the group $h\Phi(\Z^2)h^{-1}$ is a group of translations by a elementary geometric argument, finishing the proof in this case (Proposition \ref{rotzero}). The case where $|\lambda_A| = |\lambda_B|$ is more difficult and requires further understanding of the rotation set and a modification of Proposition \ref{rotzero}.

\textbf{Step 4}: The case where $|\lambda_A| < |\lambda_B|$: This case is much more difficult, we proceed by contradiction and assume such an action exists. By using the classification of the affine actions of $\Gamma_B$ (Section \ref{SSAffine}), we can show that the set of rotation vectors is finite and passing to a finite index subgroup we assume that for all $v \in \Z^2$, $\Phi(v)$ preserves set-wise each of the leaves of the unstable foliation of $\Phi(A)$, obtaining  actions of $\Z^2$ in copies of $\R$ (the leaves of the foliation). We then proceed to study such actions and show that away from fixed points, such actions are conjugated to an action by translations in $\R$ that satisfies  a specific algebraic (diophantine) condition on the possible translations. This allows to use one-dimensional Herman-Yoccoz theory to embed the action of $\Z^2$ into a flow $g_t$  that preserves the leaves of the unstable foliation. The vector field $X$ generating such flow must satisfy $D(\Phi(B))X = \lambda_B X$ by the group relations of $\Gamma_B$ and then if $\mu$ is the SRB measure preserved by $f$, we must have that $h_{\mu} (f) = \log(|\lambda_B|)$, but as the topological entropy $h_{top}(\Phi(B)) = \log(|\lambda_A|)$ being $\Phi(B)$ conjugate to $A$ we must have that $\log|\lambda_A| \geq  \log|\lambda_B|$, contradicting the assumption that $|\lambda_A| < |\lambda_B|$.

We finish this section by discussing the classification of affine actions of our group $\Gamma_B$.
\subsection{Classification of affine actions}\label{SSAffine}
Let us consider affine actions $\Phi: \Gamma_B \to \text{Aff}(\T^2)$ of the form $\Phi(B)=A\in \SL_2(\Z)$, and $\Phi(e_i): x\mapsto x+\rho_i\in \R^2$ where $e_1=(1,0)$ and $e_2=(0,1)$, $i=0,1.$ We denote by $\rho$ the matrix $(\rho_1,\rho_2)$ formed by $\rho_1$ and $\rho_2$ as columns. For this affine action, the group relation can be written as (equation \eqref{EqCommute})
\begin{equation*}
A\rho=\rho B \mathrm{\ mod} \ \Z^{2\times 2}.
\end{equation*}
The main result of this section is the following. 

\begin{proposition}\label{LmNull} Let $A$ and $B$ be two matrices in $\SL_2(\Z)$ satisfying \eqref{EqCommute}. Then 
\begin{enumerate}
\item If $tr(A)\neq tr(B)$, then all solutions $\rho$ of the equation \eqref{EqCommute} have rational coefficients and therefore the action $\Phi$ is not faithful.
\item If $tr(A)=tr(B)$,  and $|\mathrm{tr}(A)|>2$, then 
\begin{enumerate}
\item the null space of the map $X\mapsto AX-XB$ is spanned by $u_A\otimes u_{B^t}$ and $u_{A^{-1}}\otimes u_{B^{-t}}$. So the general solution to \eqref{EqCommute} has the form
$$\rho=\rho_C+c_1 u_A\otimes u_{B^t}+c_2 u_{A^{-1}}\otimes u_{B^{-t}},\ c_1,c_2\in \R $$
where $\rho_C\in \Q^{2\times 2}$ is a particular  solution to $A\rho-\rho B=C,\ C\in \Z^{2\times 2}.$
\item Suppose in addition that $A=B$, then the null space of $X\mapsto AX-XA$ is spanned by $\{\mathrm{id}, A\}$. So the general solution to \eqref{EqCommute} has the form
$$\rho=\rho_C+c_1 \mathrm{id}+c_2 A,\ c_1,c_2\in \R.$$

\end{enumerate}
\end{enumerate}
\end{proposition}

We use the following definition of Kronecker product.
\begin{definition}
Let $A'=(a_{ij})$ be a $m\times n$ matrix and $B'$ be a $p\times q$ matrix. We define the {\it Kronecker product} $A'\odot  B'$ as the $mp\times nq$ matrix written in block form whose $(i,j)$-th block is $a_{ij}B'$ where $1\leq i\leq m,1\leq j\leq n$.
\end{definition}
The Kronecker product has the following properties:
\begin{proposition}
\begin{enumerate}
\item If $A',B'$ are both $n\times n$ with eigenvalues $\{\lambda_i\}$ and $\{\mu_i\}$,\ $i=1,\ldots,n$, respectively,  then the eigenvalues of $A'\odot B'$ are $\{\lambda_i\mu_j\}$, $i,j=1,\ldots,n.$
\item$ (A'\odot  B')(C'\odot  D')=A'C'\odot  B'D'.$
\item$(A'\odot  B')^{-1}=A'^{-1}\odot  B'^{-1}$, if both $A'$ and $B'$ are invertible.
\item (\cite{HJ}, Chapter 4.4) an $n\times n$ matrix $X$ solves $A'X+XB'=C'$ for $A',B',C'$ $n\times n$ matrices, if and only if the $\R^{n^2}$ vector $\vec X$ formed by listing columns of $X$ solves the equation $(\mathrm{id}_{n}\odot  A'+B'\odot \mathrm{id}_n)\vec X=\vec C$.
\end{enumerate}
\end{proposition}
\begin{proof}[Proof of Proposition \ref{LmNull}]
Now for $A,B\in \SL_2(\Z)$,  there are two cases depending on whether $tr(A)$ and $tr(B)$ are equal or not.

Suppose $tr(A)\neq tr(B)$. In this case, the set of eigenvalues of $A$ has no intersection with that of $B$, so that $\mathrm{id}_{2}\odot A-B\odot \mathrm{id}_2$ is nondegenerate. The equation $A\rho=\rho B$ has only trivial solution and the equation $A\rho=\rho B+C$ where $C\in \Z^{2\times 2}$ has only rational solution. In this case, the affine action generated by \eqref{EqAffine} cannot be faithful.

Suppose $tr(A)=tr(B)$. Nontrivial affine actions exist already in the case of $C=0$. Indeed, in this case, we have
$$(\mathrm{id}_{2}\odot A-B\odot \mathrm{id}_2)=(\mathrm{id}_{2}\odot A)(\mathrm{id}_{4}-B\odot A^{-1}).$$
The matrix $(\mathrm{id}_{2}\odot A)$ is nondegenerate and the matrix $(\mathrm{id}_{4}-B\odot A^{-1})$ is degenerate since $B\odot A^{-1}$ has two eigenvalues 1. So the equation $A\rho=\rho B$ has nontrivial solution.

Note that in general we may have $B$ and $A$ with $tr(A)=tr(B)$ without being conjugate in $SL_2(\Z)$. For example one can take $B=\left(\begin{array}{cc}
1& 2\\
1&3
\end{array}\right)$ and $A=\left(\begin{array}{cc}
2& 1\\
3&2
\end{array}\right)$, but we must point out that for a matrix $A$ in $SL_2(\Z)$, the number of conjugacy classes of matrices $B$ with $tr(A)= tr(B)$ is finite and it is related to the class number of the number field defined by the eigenvalues of $A$.

We next prove item (2) in the statement. By assumption, $A$ and $B$ have same eigenvalues $\{\lambda,1/\lambda\}$ with $\lambda\neq 1/\lambda$. In this case, the null space of $X\mapsto AX-XB$ has dimension 2 by Theorem 4.4.14 of \cite{HJ}. It is clear that both $u_A\otimes u_{B^t}$ and $u_{A^{-1}}\otimes u_{B^{-t}}$ are in the null space and are linearly independent, so we obtain item (2.a).

For  item (2.b),  again the null space has dimension 2 and it is clear that $\mathrm{id}$ and $A$ are in the null space and are linearly independent (see also Corollary 4.4.15 of \cite{HJ}). This gives item (2.b). 

To show that $\rho_C$ can be chosen to be rational, we apply Gauss elimination to the equation 
$(\mathrm{id}\odot  A-B\odot \mathrm{id})\vec \rho=\vec C$. This completes the proof. 
\end{proof}



\end{section}

\begin{section}{Basic reductions and Rotation set properties}\label{secrota}

In this section we begin the proof of Theorem \ref{ABC3} by first making some reductions and then proceed to study the rotation set of the translation subgroup of $\Gamma_B$. This carries out Step 1 and Step 2 of the outline discussed in Section \ref{SABC0}.

\subsection{The $\Z^2$ action has trivial homotopy type}\label{SSHomotopy}
To start with the proof of Theorem \ref{ABC3}, we have from Theorem \ref{main2}, the following corollary:

\begin{proposition}\label{ABCeasy}

Suppose that $B \in \SL_2(\Z)$ is Anosov and $\Phi: \Gamma_B \to \Diff(\T^2)$ is such that the diffeomorphism $f:= \Phi(B)$ is an Anosov diffeomorphism of $\T^2$. Then there exist a finite index subgroup $\Gamma' \subset \Gamma_B$ isomorphic to $\Gamma_B$ and  $g \in \mathrm{Homeo}(\T^2)$ such that the conjugate action $\Phi': \Gamma' \to \mathrm{Homeo}(\T^2)$ given by $\Phi'(\gamma) := g\Phi(\gamma)g^{-1}$ satisfy:

\begin{enumerate}
\item  $A =  g\Phi(B)g^{-1}\in \SL_2(\Z)$.
\item Each $\Phi'(\gamma),\ \gamma \in \Gamma',$ preserves the unstable foliation $\FF$ of $A$ $($up to replacing $A$ with $A^{-1})$.
\item For each $\gamma \in \Z^2 \cap \Gamma'$, $\Phi'(\gamma)$ is homotopic to the identity.
\end{enumerate}

\end{proposition}

\begin{proof}

By Franks Theorem \cite{F},  there is a homeomorphism $g \in \T^2$ homotopic to the identity such that the conjugate $g\Phi(B)g^{-1}$ is an Anosov linear map of $\T^2$ homotopic to the action $\Phi(B)_*$ of $\Phi(B)$ on $H_1(\T^2,\Z)$. We denote by $A=\Phi(B)_*\in \SL_2(\Z)$. 

By Theorem \ref{main2}, there exists a finite index subgroup of $\Gamma_1 \subset \Gamma_B$ that preserves one of the stable or unstable foliations of $A$ denoted by $\FF$, moreover we can assume that such subgroup $\Gamma_1$ contains $B$, because $\Phi(B)$ preserves $\FF$.

Now observe that $\Phi$ gives a homomorphism: $$\Psi:  \Gamma_B \to \mathrm{MCG}(\T^2)$$ where $\mathrm{MCG}(\T^2)$ is the mapping class group of $\T^2$. It is well known that $\mathrm{MCG}(\T^2)$ is isomorphic to $\mathrm{SL}_2(\Z)$ and $\mathrm{PSL}_2(\Z)$ is isomorphic to the free product $\mathbb{Z}_2 *\mathbb{Z}_3$ which is virtually free. As $\Gamma_B$ is a solvable group and $\Psi(B)$ has infinite order (because $\Phi(B)$ is Anosov), by using that $B$ normalizes $\Z^2$ and $\mathrm{MCG}(\T^2)$ virtually free, it follows easily that $\Psi(\Z^2)$ must be finite and so we can find a subgroup of the form $\Z \ltimes_B k\Z^2$ for some $k \neq 0$ such that $\Psi(k\Z^2)$ is trivial, we call this group $\Gamma_2$.

The group  $\Gamma_3 = \Gamma_1 \cap \Gamma_2$ satisfy both (1) and (2) and contains a subgroup $\Gamma'$ of the form  $\Z \ltimes_B k'\Z^2$ for some $k' \neq 0$, which is isomorphic to $\Gamma_B$.
\end{proof}
\subsection{Rotation sets}\label{SSRotation}

We now discuss rotation sets for toral homeomorphisms and some of its properties.

\begin{definition} For a homeomorphism $f: \T^n \to \T^n$ homotopic to the identity of $\T^n = \R^n/\Z^n$ and a lift $\tilde{f}: \R^n \to \R^n$ of $f$ the \emph{Misiurewicz-Ziemian rotation set} $\rot(\tilde{f})$ is defined to be the subset of $\R^n$ consisting of limits of sequences of the form $$\bigg\{\frac{\tilde{f}^{n_i}(x_i) - x_i}{n_i} \bigg\}$$ where $x_i \in \R^n$ and $n_i \to \infty$.
\end{definition}

The rotation set is an important dynamical invariant and it is known to be compact for arbitrary $n$ and  convex in the case $n = 2$, see \cite{MZ}. Observe also that the rotation set depends on the lift $\tilde{f}$ but it is defined for $f$ up to translation by an element of $\Z^n$. Observe also that if $\rot(\tilde{f})$ consists of only one point, we have that for every $x, y \in \R^n$:
\begin{equation}\label{eq1}
\lim_{n \to \infty} \frac{({\tilde{f}}^n(x) - x) - ({\tilde{f}}^n(y) - y)}{n} = 0
\end{equation}
and moreover such limit converges uniformly with respect to $x$ and $y$ in $\R^n$.

\begin{definition}
For an $f$-invariant probability measure $\mu$ on $\T^n$ and a lift $\tilde{f}: \R^n \to \R^n$, one can also define the \emph{average rotation vector} $\rot_{\mu}(\tilde{f})$ by $$\rot_{\mu}(\tilde{f}) = \int_{\T^n}  \tilde{f}(\tilde{x}) - \tilde{x} \ d\mu.$$
\end{definition}

When the measure $\mu$ is ergodic, Birkhoff's theorem implies that $\rot_{\mu}(\tilde{f}) \in \rot(\tilde f)$ and so in the case when  $n = 2$, the convexity of $\rot(\tilde f)$, implies that for any invariant probability measure $\mu$, $\rot_{\mu}(\tilde{f}) \in \rot(\tilde f)$. We will use the following fact repeatedly:

\begin{lemma}
Let $f: \T^2 \to \T^2$ be a homeomorphism  homotopic to the identity and $\tilde{f}: \R^2 \to \R^2$ be a lift of $f$. Let $r$ be an extreme point of the convex set $\rot(\tilde{f})$. Then  there exists an invariant ergodic probability measure $\mu$ on $\T^2$ such that $\rot_{\mu}(\tilde{f}) = r$.
\end{lemma}

\subsection{Trivializing the rotation sets}\label{SSRSTrivial}

The main result of this section is the following.

\begin{proposition}\label{difficult}

Suppose $B \in \SL_2(\Z)$ is Anosov and $\Phi: \Gamma_B \to \mathrm{Homeo}(\T^2)$ is such that:

\begin{enumerate}
\item $A = \Phi(B)$ is an Anosov linear map.
\item Let $e_1, e_2$ two generators of $\Z^2 \subset \Gamma_B$. Assume $\Phi(e_1)$ and $\Phi(e_2)$ are homotopic to the identity and  preserve the unstable foliation $\FF$ of $A$.
\end{enumerate}

Then if $tr(A) \neq tr(B)$, the rotation set of $\Phi(e_i),\ i=1,2,$ consists of a single point.

\end{proposition}

We proceed to prove this proposition. We begin by showing that the rotation set of $\Phi(e_i),\ i=1,2,$ is either a segment of a line or a point. Let $\bu \in \R^2$ be the expanding eigenvector of $A$, we have the following:

\begin{proposition}\label{segment} For each $v \in \Z^2$, the rotation set of a lift $f_v$ of $\Phi(v)$ is a segment of the form
$$S = \big\{ \alpha_v  + t\bu \text{, where } t \in [l_v^{-}, l_v^{+}] \big\}$$ where $\alpha_v \in \R^2$ and $l_v^{-},l_v^{+} \in \R$.
\end{proposition}

\begin{proof}

Recall that $A$ is a linear anosov map the unstable foliation  $\FF$ is linear. Using that $f_v$ preserves the unstable foliation $\FF$ of $A$, we have that $f_v(\ell_0) = \ell_0 + \alpha_v$ for some $\alpha_v \in \R^2$ and some leaf $\ell_0 \in \FF$. Therefore, we have $f_v(\ell_0 + w) = \ell_0 + w +\alpha_v$ for every $w\in \Z^2$ because $f_v$ commutes with $\Z^2$ translations. As for the (irrational) foliation $\mathcal{F}$ the leaves $\ell_0 + w$ where $w \in \Z^2$ are dense in $\R^2$ we conclude that $f_v(\ell) = \ell + \alpha_v$ for every $\ell \in \FF$.

Therefore for every $x \in \R^2$ we have that  $f^n_v(x) - x -  n\alpha_v$ must lie in the line $\ell' = \{t\bu , t \in \R\}$. This together with the convexity of $\rot(f_v)$ imply Proposition \ref{segment} easily.
\end{proof}

Moreover we have the following:

\begin{proposition} The following dichotomy holds:

\begin{enumerate}\label{lastfix}
\item For every $v \in \Z^2\setminus\{0\}$,  the rotation set of a lift $f_v$ of $\Phi(v)$ is a non-trivial segment.
\item There is a finite index subgroup $L$ of $\Z^2$ such that for every $v \in L\setminus\{0\}$,  the rotation set of a lift $f_v$ of $\Phi(v)$ is a point.
\end{enumerate}

\end{proposition}

\begin{proof}
Assume that for some $w \in \Z^2\setminus\{0\}$,  the rotation set of a lift $f_w$ of $\Phi(w)$ is a point. We will show in that case that there is a finite index subgroup $L<\Z^2$ such that for every $v \in L\setminus\{0\}$, the rotation set of a lift $f_v$ of $\Phi(v)$ is a point. 

Observe that for every $v \in \Z^2$ and any integer $n$,  the rotation set of a lift of $f_v$ of $\Phi(v)$ consists of a single point if and only if the rotation set of  $f_{nv}$ consists of a single point. 

As $\Phi(B(w))$ is conjugate to $\Phi(w)$ by $A$ using the group relation in $\Gamma_B$, the rotation set of a lift of $\Phi(B(w))$ must also consist of a single point. Observe that as $B$ is Anosov, the vectors  $w$ and $B(w)$ generate a group $L$ of finite index in $\Z^2$, and so it is enough to show that the rotation set of $\Phi(v)$ for any $v \in L\setminus\{0\}$ consists of a single point.

Therefore, let  $v = nw + mB(w)$, then each of the rotation sets of $f_{nw}$ and $f_{-mB(w)}$ consists of a point, so does the rotation set of  $f_{nw + mB(w)}$,  because if it was a segment, then we obtain a contradiction to the following:

\begin{proposition}\label{sumsegment}
Let $v,w \in \Z^2$, if the rotation set of a lift $f_v$ of $\Phi(v)$ consists of a single point and the rotation set of $f_w$ is a non-trivial segment, then the rotation set of a lift $f_{v+w}$ of $\Phi(v + w)$ is a non-trivial segment.
\end{proposition}

\begin{proof}

Let $\rot(f_v) = \{\alpha\}$, and observe that $f_vf_w$ is a lift of $\Phi(v+w)$. By the definition of rotation set, if $\beta_1$ and $\beta_2$ are two different rotation vectors for $f_w$, we have two sequences  $\{x_i\}, \{y_i\}$ in $\T^2$ and two divergent sequences $\{n_i\}$ and $\{m_i\}$ of positive integers such that 
$$ \lim \frac{{f_w}^{n_i}(x_i) - x_i}{n_i} = \beta_1 \ \ \text{ and } \ \ \lim \frac{{f_w}^{m_i}(y_i) - y_i}{m_i}  = \beta_2.$$

Moreover we have that for every $x \in \R^2$, $\frac{{f_v}^{n}(x) - x}{n} = \alpha$ and this convergence must be uniform for all $x \in \R^2$. It follows that 

$$ \lim \frac{{(f_vf_w)}^{n_i}(x_i) - x_i}{n_i} = \lim \frac{f_v^{n_i}f_w^{n_i}(x_i) - f_w^{n_i}(x_i)}{n_i} +  \lim \frac{{f_w}^{n_i}(x_i) - x_i}{n_i} = \alpha + \beta_1,$$ where the first limit holds because of the uniform convergence and similarly $\lim \frac{{(f_vf_w)}^{m_i}(y_i) - y_i}{m_i}  = \alpha + \beta_2$ and therefore the rotation set of $f_vf_w$ has two different vectors and by Proposition \ref{segment} must be a non-trivial segment.

\end{proof}

\end{proof}

We define the following subset of the set of $2\times2$ real matrices $\mathcal{M}_{2\times2}(\R)$ whose study will be crucial to finish the proof of proposition \ref{difficult}.

\begin{definition}\label{jointrota} For a fixed choice of lifts $f_{e_1}$ and $f_{e_2}$, the set of {\it joint rotation numbers} $\Sr \subset \mathcal{M}_{2\times2}(\R)$ is defined by:
$$\Sr := \bigg\{\big(\rot_{\mu}(f_{e_1}), \rot_{\mu}(f_{e_2}) \big) \in \mathcal{M}_{2\times2}(\R)  \ \bigg| \  \mu \text{ is } \Z^2\text{-invariant} \bigg\}.$$
\end{definition}

Observe that $\Sr$ must be contained in the bounded square $\rot(f_{e_1}) \times \rot(f_{e_2}) \subset \mathcal{M}_{2\times2}(\R)$ since a choice of lifts $f_{e_1}$ and $f_{e_2}$ is fixed.  Recall that for a subset $X$ of  $\mathcal{M}_{2\times2}(\R)$, $ X-X: = \{x-y\ |\ x,y \in X\}$.

Proposition \ref{difficult} follows easily from  the following  key proposition.

\begin{proposition}\label{properties} The set $\Sr$ satisfies the following properties:
\begin{enumerate}
\item $\Sr$ is a convex and compact set.
\item For every $n \in \Z$: if $R_1, R_2 \in \Sr$, then $A^n(R_1 - R_2)B^{-n}  \in \Sr - \Sr$.
\item One of the following holds:

\begin{enumerate}
\item The set $\rot(f_{e_1}) \times \rot(f_{e_2})$ consists of a single point.
\item $|\lambda_A|= |\lambda_B|$ and $\Sr$ is one of the diagonals of $\rot(f_{e_1})\times \rot(f_{e_2})$.
\end{enumerate}
\end{enumerate}
\end{proposition}

We postpone the proof of Proposition \ref{properties} to the next subsection.

\subsection{Description of the rotation sets}
In this section, we give the proof of Proposition \ref{properties}.
\begin{lemma}The action of $\Z^2$ on $\T^2$ lifts to an action on the universal covering $\R^2$ of $\T^2$. In other words, for every $v \in \Z^2$, we can choose a lift $f_v: \R^2 \to \R^2$ of $\Phi(v)$, so that we have \begin{equation}\label{linearity}
f_{v+w} = f_vf_w,\text{ for all  }v,w \in \Z^2.
\end{equation}

\end{lemma}
\begin{proof}If $v, w \in \Z^2$ and $f_v, f_w$ are lifts of $\Phi(v), \Phi(w)$ respectively, we have that  $f_vf_wf_v^{-1}f_w^{-1}$ is a lift of the identity so it must be a translation by some $k \in \Z^2$, the action of $\Z^2$ lifts if and only if $k = 0$ for all $v,w$.

From the proof of Proposition \ref{segment}, we have $f_v(\ell)=\ell+\alpha_v$ for all $\ell\in \FF$ and $\alpha_v\in \R^2$ depending only on $v\in \Z^2$, so we have $f^{-1}_v(\ell)=\ell-\alpha_v$. This implies  that $f_vf_wf_v^{-1}f_w^{-1}(\ell)=\ell$ for all $\ell \in \FF$. In particular, $f_vf_wf_v^{-1}f_w^{-1}$ preserves the leaf $\ell(0)\in \FF$ passing through $0$. However, as the foliation is irrational $\ell(0) \cap \Z^2 = \{0\}$, which implies that $k = 0$.
\end{proof}

In the following discussion we fix such choice of lifts.


For a $\Z^2$-invariant probability measure $\mu$ on $\T^2$ and each $i = 1,2$, we consider the vector $\rot_{\mu}(f_{e_i})$ as a column vector and so we construct a $2\times2$ matrix in $\mathcal{M}_{2\times2}(\R)$ given by: $$\big(\rot_{\mu}(f_{e_1}), \rot_{\mu}(f_{e_2}) \big).$$

Observe also that as $ A^{-1}\Phi(e_i)A =  \Phi(B^{-1}(e_i))$, then $A^{-1}f_{e_i}A$ is also a lift of $\Phi(B^{-1}(e_i))$ and so there exists $w_i \in \Z^2$ such that  
\begin{equation}\label{ws}
f_{e_i} \circ A = T_{w_i} \circ A \circ f_{B^{-1}(e_i)}.
\end{equation}

One can easily check that for any $C \in \SL_2(\Z)$, the linearity equation \eqref{linearity} implies that:
\begin{equation}\label{rotmatrix}
(\rot_{\mu} f_{Ce_1}, \rot_{\mu} f_{Ce_2}) = (\rot_{\mu} f_{e_1}, \rot_{\mu} f_{e_2})C.
\end{equation}

We have the following:

\begin{proposition}\label{conjugarot}

If $\mu$ is a $\Z^2$-invariant probability measure on $\T^2$, then $A_{*}\mu$ is $\Z^2$-invariant and we have the following:
$$ \big( \rot_{A_{*}\mu}(f_{e_1}), \rot_{A_{*}\mu}(f_{e_2}) \big)= A\big(\rot_{\mu}(f_{e_1}), \rot_{\mu}(f_{e_2}) \big)B^{-1} + \big( w_1, w_2 \big).$$

where $w_1,w_2 \in \Z^2$ are as in equation \eqref{ws}.
\end{proposition}

\begin{proof}
We first show that $A_*\mu$ is $\Z^2$-invariant. Since $\mu$ is $\Z^2$ invariant, we have $\Phi(v)_*\mu=\mu$ for all $v\in \Z^2$. So we get $$\Phi(v)_*(A_*\mu)(E)=A_*\mu(\Phi(v)^{-1}(E))=\mu(A^{-1}\Phi(v)^{-1}(E))= \mu(\Phi(B^{-1}v)A^{-1}(E)) = \mu(A^{-1}(E))$$ for all Borel set $E$ and all $v\in \Z^2$. This gives the $\Z^2$-invariance of $A_*\mu$.

We next prove the identity. By definition we have for every $i=1,2$:
\begin{align*}
\rot_{A_{*}\mu}(f_{e_1}) &= \int_{\T^2}  f_{e_i}(\tilde{Ax}) - \tilde{Ax} \ d\mu \\
&= \int_{\T^2}  f_{e_i}(A \tilde{x}) - A\tilde{x} \ d\mu \\
&= \int_{\T^2}   T_{w_i} \circ Af_{B^{-1}(e_i)} (\tilde{x}) - A(\tilde{x}) \ d\mu \\
&= \int_{\T^2}   A \big( f_{B^{-1}(e_i)} (\tilde{x}) - \tilde{x}) \big) + w_i \ d\mu \\
&= A \int_{\T^2}  (f_{B^{-1}(e_i)}(\tilde{x}) - \tilde{x}) \ d\mu + w_i\\
&= A(\rot_{\mu} f_{B^{-1}e_i}) + w_i.
\end{align*}

And so to finish the proof, we must only show that:
$$(\rot_{\mu} f_{B^{-1}e_1}, \rot_{\mu} f_{B^{-1}e_2}) = (\rot_{\mu} f_{e_1}, \rot_{\mu} f_{e_2})B^{-1},$$
which  follows immediately from equation \eqref{rotmatrix}.

\end{proof}




\begin{proof}[Proof of Proposition \ref{properties}]

Item (1) follows from the convexity and compactness of $\Z^2$ invariant probability measures and the fact that the map $\rot_{\mu}(.)$ depends linearly in $\mu$. Item (2) follows from Proposition \ref{conjugarot}.  We remark that the integer vectors $(w_1,w_2)$ in Proposition \ref{conjugarot} might depend on $n$ if we iterate $A^n$ and $B^n$. However, they disappear after taking difference $\Sr - \Sr$.

We consider item (3). Suppose that $r \in \R^2$ is an extreme point of $\rot(f_{e_1})$, we will show that there exist $(r,s) \in \Sr$ for some $s \in \R^2$. Let $\mu_1$ a probability measure   which is $\Phi(e_1)$-invariant on $\T^2$ such that $\rot_{\mu}(f_{e_1}) = r$, if we consider the sequence of measures $$\nu_n := \frac{1}{n}\sum_{i=1}^n \Phi(e_2)_{*}^i(\mu_1)$$ and $\nu$ a limit measure of a subsequence of $\nu_n$, then $\nu$ is $\Z^2$-invariant and we have $\rot_\nu(f_{e_1})=r$. Moreover, if we denote $s = \rot_{\nu}(f_{e_2})$, then we have that $(r,s) \in \Sr$ by definition.

This implies that $\Sr$ has at least a point in two opposite sides of the square $\rot(f_{e_1}) \times \rot(f_{e_2})$, arguing similarly for $\rot(f_{e_2})$ one can show there are points in the other pair of opposite sides of the square $\rot(f_{e_1}) \times \rot(f_{e_2})$. In conclusion the set $\Sr$ contains points in each side of the square $\rot(f_{e_1}) \times \rot(f_{e_2})$. From the convexity of $\Sr$ it follows that either $\Sr$ is a one of the diagonals of  $\rot(f_{e_1}) \times \rot(f_{e_2})$ or it has non-empty interior and then item (3) follows from the following Proposition \ref{weird}.

\end{proof}

\begin{proposition}\label{weird} {The set $\Sr$ has empty interior. Moreover if $|\lambda_A| \neq |\lambda_B|$, the set $\rot(f_{e_1}) \times \rot(f_{e_2})$ consists of a single point.}
\end{proposition}

\begin{proof}

Let's show that $\Sr$ has empty interior. Assume by contradiction that $\Sr$ has non empty interior. As we know that for every $i=1,2$, points in the set $\rot(f_{e_i})$ are of the form  $\alpha_{e_i}  + t\bu$, where $t \in [l_{e_i}^{-}, l_{e_i}^{+}] $. If $\Sr$ has non-empty interior we conclude there exists $\e>0$ such that for any $t_1, t_2 \in [-\e,\e]$ we have that $(t_1\bu,t_2\bu) \in \Sr - \Sr$ and so by item (2) in Proposition \ref{properties} we have $A^n(t_1\bu, t_2\bu)B^{-n} \in \Sr - \Sr$. We will show that this cannot be true for $n$ large enough.

Diagonalizing the matrix $B$, we have that $B= Q\text{diag}(\lambda_B, \lambda_B^{-1})Q^{-1}$ for $Q = (v_1, v_2)$ where $v_1$ and $v_2$ are the eigenvectors of $B$. We then have:
\begin{align*}
A^n(t_1\bu, t_2\bu)B^{-n} &= (t_1A^n(\bu), t_2A^n(\bu))B^{-n}\\
&= \lambda_A^n(t_1\bu, t_2\bu)B^{-n}\\
&= \lambda_A^n(t_1\bu, t_2\bu)Q\text{diag}(\lambda_B^{-n},  \lambda_B^{n})Q^{-1}\\
&= (t_1\bu, t_2\bu)(v_1, v_2)\text{diag}((\lambda_A\lambda_B^{-1})^n, (\lambda_A\lambda_B)^{n})Q^{-1}\\
&= (t_1\bu, t_2\bu)((\lambda_A\lambda_B^{-1})^nv_1, (\lambda_A\lambda_B)^nv_2)Q^{-1}.
\end{align*}

One can show easily that the vector $(t_1\bu, t_2\bu)( (\lambda_A\lambda_B)^nv_2)$ diverges for an appropriate choice of $t_1, t_2 \in [-\e,\e]$ and so $A^n(t_1\bu, t_2\bu)B^{-n}$ cannot lie in the compact set $\Sr - \Sr$ for every $n >0$. Therefore $\Sr$ must have empty interior.

In the case where $|\lambda_A| \neq |\lambda_B|$. As we showed that $\Sr$ must be one of the diagonals of  $\rot(f_{e_1}) \times \rot(f_{e_2})$ we must only show that $\Sr$ does not consists of a single point. Arguing by contradiction as before if $\Sr$ was not a single point, then $(t_1\bu,t_2\bu) \in \Sr - \Sr$ for some $t_1, t_2 \in \R$ where at least one of them is different than zero. From the calculations above we have:
$$A^n(t_1\bu, t_2\bu)B^{-n} = (t_1\bu, t_2\bu)((\lambda_A\lambda_B^{-1})^nv_1, (\lambda_A\lambda_B)^nv_2)Q^{-1},$$
 which implies in particular that $(t_1\bu, t_2\bu)( (\lambda_A\lambda_B)^nv_2)$ is bounded and so by taking $n$ large we have $(t_1, t_2)$ is orthogonal to $v_2$. We also have that  $(t_1\bu, t_2\bu)( (\lambda_A\lambda_B^{-1})^nv_1)$ is bounded and so taking $n$ or $-n$ large depending whether $|\lambda_1| > |\lambda_2|$ or not, we have that $(t_1, t_2)$ is also orthogonal to $v_1$ and so $t_1 = t_2 = 0$, a contradiction.

\end{proof}

\end{section}

\begin{section}{Case $|\lambda_A| > |\lambda_B|$.}\label{SABC1}

In this section, we  prove Theorem \ref{ABC3} in the case where $|\lambda_A|  > |\lambda_B|$, more concretely, we show the following proposition under the assumption that $|\lambda_A|  > |\lambda_B|$.

\begin{theorem}\label{ABC4}

Suppose that  $B \in \SL_2(\Z)$ is Anosov and $\Phi: \Gamma_B \to \Diff(\T^2)$ is such that the diffeomorphism $\Phi(B)$ is an Anosov diffeomorphism of $\T^2$ homotopic to $A\in \SL_2(\Z)$  with $|\lambda_A|\geq |\lambda_B|$.

Then $\Phi$ is topologically conjugate to an affine action of $\Gamma_B$ of the form \eqref{EqAffine} up to finite index. More concretely, there exist a finite index subgroup $\Gamma' \subset \Gamma_B$ and  $g \in \mathrm{Homeo}(\T^2)$ such that $g\Phi(\Gamma')g^{-1}$ coincides with an action of the form  \eqref{EqAffine}.

\end{theorem}


Theorem \ref{ABC4} in the case $|\lambda_A| > |\lambda_B|$ follows from Proposition \ref{difficult} and the following proposition. The case in Theorem \ref{ABC4} when $|\lambda_A| = |\lambda_{B}|$ will be consider in the next section. 

\begin{proposition}\label{rotzero} Suppose $\Gamma_B = \Z \ltimes \Z^2$ is an ABC group as in Theorem \ref{ABC4} and $\Phi: \Gamma_B \to \mathrm{Homeo}(\T^2)$ is such that:

\begin{enumerate}
\item $A= \Phi(B)$ is an Anosov linear map with $|\lambda_A|\geq |\lambda_B|$.
\item $\Phi(e_1)$ and $\Phi(e_2)$ are homotopic to the identity.
\item For $i=1,2$, if $f_{e_i}: \R^2 \to \R^2$ is any lift of $\Phi(e_i)$,  the rotation set $\rot(f_{e_i})$ consists of a single point.
\end{enumerate}
Then $\Phi$ is an affine action.
\end{proposition}

\begin{proof}
It is enough to show that $f_{e_i},\ i=1,2,$ is a translation of $\R^2$. Indeed, for any $v=(v_1,v_2)\in \Z^2$, we have $f_v$, the lift of $\Phi(v)$, equals $f^{v_1}_{e_1}f_{e_2}^{v_2}+k,$ for some $k\in \Z^2$. Since $f_{e_i}$ is a translation, we get that $f_v$ hence $\Phi(v)$ are also translations.

Suppose that $x,y \in \R^2$ are such that $f_{e_1}(x) = x + p$ and $f_{e_1}(y)  = y + q$ for some $p,q \in \R^2$, we must show that $p =q$. We have the following:
\begin{align*}
A^nf_{e_1}(x) = A^n(x) + A^n(p),\\
A^nf_{e_1}(y) = A^n(y) + A^n(q).
\end{align*}
We have that both $A^nf_{e_1}A^{-n}$ and $f_{B^n(e_1)}$ are lifts of $\Phi(B^n(e_1))$, therefore they differ by a translation and so $A^nf_{e_1} = T_wf_{B^n(e_1)}A^n$ for some translation $T_w$ for some $w \in \Z^2$, and so:
\begin{align*}
A^nf_{e_1}(x) =  f_{B^n(e_1)}A^n(x) + w,\\
A^nf_{e_1}(y) =  f_{B^n(e_1)}A^n(y) + w.
\end{align*}
Therefore from the previous equations: $$f_{B^n(e_1)}A^n(x) - f_{B^n(e_1)}A^n(y) = A^n(x) + A^n(p) - (A^n(y) + A^n(q))$$ and so we obtain:

\begin{equation}\label{eq2}
\frac{ \big( f_{B^n(e_1)}A^n(x) - A^n(x) \big) - \big( f_{B^n(e_1)}A^n(y) - A^n(y) \big)}{|B^n(e_1)|}=  \frac{A^n(p - q)}{|B^n(e_1)|}.
\end{equation}
Since we have $f_v(x+k)=f_v(x)+k,\ k\in \Z^2$ by the definition of a lift, in the expression $\big( f_{B^n(e_1)}A^n(x) - A^n(x) \big)$, we may assume that $A^n:\ \T^2\to \T^2$. Suppose
$B^n=\begin{pmatrix} a_n &  b_n \\ c_n & d_n \end{pmatrix}$, then $B^n(e_1)=(a_n,c_n)$ and $f_{B^n(e_1)}=f_{e_1}^{a_n}f_{e_2}^{c_n}$. Denoting $z_n=A^n(x) $ or $A^n(y)\in \T^2$, we have
\begin{equation*}
\begin{aligned}
f_{B^n(e_1)}(z_n)-z_n&=f_{e_1}^{a_n}f_{e_2}^{c_n}(z_n)-z_n\\
&=f_{e_1}^{a_n}f_{e_2}^{c_n}(z_n)-f_{e_2}^{c_n}(z_n)+f_{e_2}^{c_n}(z_n)-z_n\\
&=(f_{e_1}^{a_n}\Phi(c_ne_2)(z_n)-\Phi(c_ne_2)(z_n))+(f_{e_2}^{c_n}(z_n)-z_n).
\end{aligned}
\end{equation*}
The sequences $(a_n), (c_n), (|B^n(e_1)|)$ all go to infinity with a rate proportional to $|\lambda_B|^n$.
Using the fact that $f_{e_i},\ i=1,2,$ has one point in its rotation set, by equation \eqref{eq1} and the uniformity of the convergence we get
$$\frac{1}{|B^n(e_1)|}[(f_{e_1}^{a_n}\Phi(c_ne_2)(x_n)-\Phi(c_ne_2)(x_n))-(f_{e_1}^{a_n}\Phi(c_ne_2)(y_n)-\Phi(c_ne_2)(y_n))]\to 0,$$
$$\frac{1}{|B^n(e_1)|}[(f_{e_2}^{c_n}(x_n)-x_n)-(f_{e_2}^{c_n}(y_n)-y_n)] \to 0,$$
where $x_n=A^n(x),\ y_n=A^n(y)\in \T^2$.
This shows that the LHS of \eqref{eq2} vanishes. So we have $$\lim_{n \to \infty} \frac{A^n(p - q)}{\lambda_B^n} = 0.$$ But this can happen only if $p -q$ lies in the one dimensional eigenspace of $\lambda_A^{-1}$ since we have $|\lambda_A|\geq |\lambda_B|$. Applying the same argument with $B^{-1}$ we have that $p - q$ lies in the one dimensional eigenspace of $\lambda_A$, therefore $p -q = 0$ as we wanted. This proves that $f_{e_1}$ is a translation. Similarly, we get that $f_{e_2}$ is a translation.

\end{proof}

\end{section}

\begin{section}{Case $|\lambda_A| = |\lambda_B|$}\label{mistakecorrection}

In this section we finish the proof of Theorem \ref{ABC4} dealing with the remaining case where $|\lambda_A| = |\lambda_B|$. Recall we assume that $\Phi: \Gamma_B \to \text{Homeo}(\T^2)$ is an action satisfying the following: 

\begin{enumerate}
\item $A= \Phi(B)$ is an Anosov linear map with $|\lambda_A| = |\lambda_B|$.
\item For every $v \in \Z^2$, $\Phi(v)$ is isotopic to the identity and $f_v: \R^2 \to \R^2$ denotes some lift of $\Phi(v)$.
\end{enumerate}

Observe that we cannot use Proposition \ref{rotzero} because we have not shown the rotation set of elements of $\Z^2$ consists of a single point in the case $|\lambda_A| = |\lambda_B|$. We will argue by contradiction to show that for all $v \in \Z^2$ the rotation set of $\Phi(v)$ consists of a single point. Therefore we will suppose that some $\rot(f_v)$ is not a single point and in that case, recall we have shown in Proposition \ref{segment} and \ref{lastfix} that $\rot(f_v)$ must be a non-trivial segment for every $v \in \Z^2$. Let $e_1, e_2$ be the standard basis of $\Z^2$ and recall the joint rotation set $\Sr$ defined in Section \ref{SSRSTrivial}. By Proposition \ref{properties}, we can assume that $\Sr$ is one of the diagonals of the square $\rot (f_{e_1}) \times \rot (f_{e_2})$. Moreover, a point in the segment $\Sr$ has the form $\boldsymbol\alpha+(a\mathbf u, b\mathbf u),\ \boldsymbol\alpha\in \R^{2\times 2},\ a/b\notin \mathbb Q$ arguing as Proposition \ref{weird}.  Indeed, if $a/b\in \mathbb Q$, then following the proof of Proposition \ref{weird}, we will see that $((a\mathbf u, b\mathbf u))(\lambda_A\lambda_B)^n v_2$ diverges as $n\to\infty$ since we have $(a,b)\cdot v_2\neq 0$ due to the irrationality of $v_2$. Then the argument of Proposition \ref{weird} goes through and we get $\rot(f_{e_1})\times \rot(f_{e_1})$ is a single point.  The irrationality of $a/b$ implies that the segment $\Sr$ cannot form a loop on the torus $\T^2\times \T^2(\supset \rot (f_{e_1}) \times \rot (f_{e_2}))$.

We start by showing the existence of a nice $\Gamma_B$-invariant  measure $\nu$ where the asymptotic behaviour of the action of $\Z^2$ is given by the rotation number with respect to $\nu$. We will then use this measure together with a modification of Proposition \ref{rotzero} to obtain a contradiction.

\begin{proposition}\label{nicemeasure} There exists a $\Gamma_B$-invariant probability measure $\nu$ on $\T^2$ such for any $i=1,2$, any lift $f_{e_i}: \R^2 \to \R^2$ of $\Phi(e_i)$ and $\nu$ a.e. $x \in \T^2$  we have :
$$\lim_{n \to \infty} \frac{f_{e_i}^{n}(x) - x}{n} = \rot_{\nu}(f_{e_i})  \  \mod \Z^2.$$
\end{proposition}

\begin{proof}

We will prove the proposition for the coordinate vectors $v = e_1, e_2$. 

Recall also that if $\mu$ is a $\Z^2$-invariant measure, then $A_{*}\mu$ is also $\Z^2$-invariant and we have an action of $A$ in $\Sr$ ( $\text{mod} \  \Z^2 \times \Z^2$) which is described by Proosition \ref{conjugarot} and states that:
$$ \big( \rot_{A_{*}\mu}(f_{e_1}), \rot_{A_{*}\mu}(f_{e_2}) \big)= A\big(\rot_{\mu}(f_{e_1}), \rot_{\mu}(f_{e_2}) \big)B^{-1}  \ \ \ \mod \Z^2 \times \Z^2.$$ 

Observe that the action in $\Sr$ is linear and therefore is clearly continuous. Therefore the two endpoints of the interval $\Sr$ are either fixed or permuted and by replacing $A$ with $A^2$ (this is the same as replacing $\Gamma_B$ with the finite index subgroup $\Gamma_{B^2}$) we will assume that the endpoints of $\Sr$ are fixed. We fix an endpoint of $\Sr$ and take a $\Z^2$-ergodic invariant measure $\mu$ such that $(\rot_{\mu}(f_{e_1}), \rot_{\mu}(f_{e_2}))$ is equal to this endpoint, this implies in particular that both $\rot_{\mu}(f_{e_1})$
 and $\rot_{\mu}(f_{e_2})$ are extremal points of $\rot(f_{e_1})$ and $\rot(f_{e_2})$ respectively.
 
 Consider the sequence of measures $\mu_n : = \frac{1}{n}\sum_{j = 1}^{n} A^{j}_{*} \mu.$
From the equality $$( \rot_{A_{*}\mu}(f_{e_1}), \rot_{A_{*}\mu}(f_{e_2}) \big) =  (\rot_{\mu}(f_{e_1}), \rot_{\mu}(f_{e_2})),\ \mathrm{mod}\ \Z^2\times\Z^2,$$ we get by linearity that $(\rot_{\mu_n}(f_{e_1}), \rot_{\mu_n}(f_{e_2}) = (\rot_{\mu}(f_{e_1}), \rot_{\mu}(f_{e_2}))$. We let $\nu$ be any weak $*$-limit of $( \mu_n )$. Observe that $\nu$ is $\Z^2$- and $A$-invariant and therefore is $\Gamma_B$-invariant. By continuity we also have $$(\rot_{\nu}(f_{e_1}), \rot_{\nu}(f_{e_2})) = (\rot_{\mu}(f_{e_1}), \rot_{\mu}(f_{e_2})).$$  

It is possible that $\nu$ is not $\Phi(e_1)$ ergodic, but nonetheless for $\nu$-a.e.every $\Phi(e_1)$-ergodic component $\eta$ of $\nu$ we must have $\rot_{\eta}(f_{e_1}) = \rot_{\nu}(f_{e_1})$ because $\rot_{\nu}(f_{e_1})$ is an extremal point of $\rot(f_{e_1})$. From Birkhoff ergodic theorem we conclude that: 
$$\lim_{n \to \infty} \frac{f_{e_1}^{n}(x) - x}{n} = \rot_{\nu}(f_{e_1})  \  \mod \Z^2$$ for $\eta$ a.e. $x \in \T^2$ and because this is true for a.e. every ergodic component $\eta$, it is also true for $\nu$ a.e. $x \in \T^2$ as we wanted. We can argue similarly for $\Phi(e_2)$.

\end{proof}

We will now  show that for a measure $\nu$ as in Proposition \ref{nicemeasure},  the action of  $\Z^2$ restricted to the support of $\nu$ is by translations, the proof is very similar to the proof of Proposition \ref{rotzero}. This will lead us later to a contradiction.
\begin{proposition}\label{technical} For every $i  =1,2 $ and every $\delta, \epsilon > 0$, there exists a set $C$ with $\nu(C) > 1 - \delta$ such that for every $x,y \in C$ we have: $$|(f_{e_i}(x) - x)  - (f_{e_i}(y) - y)| < \epsilon.$$
\end{proposition}
\begin{proof} The proof follows closely the argument in Proposition \ref{rotzero}. We only prove it for $v = e_1$.

Suppose that $p_x,p_y \in \R^2$ are such that $f_{e_1}(x) = x + p_x$ and $f_{e_1}(y)  = y + p_y$ for each $x,y \in \R^2$. We will find a set $C$ with $\mu(C) > 1 - \delta$ such that $|p_x-p_y| < \epsilon$ for $x,y \in C$. observe that because $\Z^2$ preserves the unstable foliation, there exists a fixed $\alpha \in \R^2$ and $t_x, t_y \in \R$ such that $p_x = \alpha + t_x\bu$ and $p_y = \alpha + t_y\bu$ (see the proof of  Proposition \ref{segment}).

We have the following:
\begin{align*}
A^nf_{e_1}(x) = A^n(x) + A^n(p_x),\\
A^nf_{e_1}(y) = A^n(y) + A^n(p_y).
\end{align*}
As in proposition \ref{rotzero}, we have that for some $w \in \Z^2$ we have:
\begin{align*}
A^nf_{e_1}(x) =  f_{B^n(e_1)}A^n(x) + w,\\
A^nf_{e_1}(y) =  f_{B^n(e_1)}A^n(y) + w.
\end{align*}

Therefore from the previous equations: $$f_{B^n(e_1)}A^n(x) - f_{B^n(e_1)}A^n(y) = A^n(x) + A^n(p_x) - (A^n(y) + A^n(p_y))$$ and if we let $x_n: = A^n(x)$ and $y_n := A^n(y)$ we obtain:

\begin{equation}\label{txty}
\bigg | \frac{ \big( f_{B^n(e_1)}(x_n) - x_n \big) - \big( f_{B^n(e_1)}y_n - y_n \big)}{\lambda_B^n} \bigg|=  \bigg|\frac{A^n(p_x - p_y)}{\lambda_B^n} \bigg| =   \bigg| \frac{A^n(t_x\bu - t_y\bu)}{\lambda_A^n}  \bigg|= |t_x - t_y|.
\end{equation}

Suppose $B^n=\begin{pmatrix} a_n &  b_n \\ c_n & d_n \end{pmatrix}$, then $B^n(e_1)=(a_n,c_n)$ and $f_{B^n(e_1)}=f_{e_1}^{a_n}f_{e_2}^{c_n}$. We can choose a large constant $M>0$ such that the quantities $\frac{|a_n|}{|\lambda_B^n|}, \frac{|b_n|}{|\lambda_B^n|}, \frac{|c_n|}{|\lambda_B^n|}, \frac{|d_n|}{|\lambda_B^n|}$ are bounded above by $M$ and below by $1/M$.

From Proposition \ref{nicemeasure}, we can choose $N_{\epsilon}$ and $C_{\delta}$ such that $\mu(C_{\delta}) > 1 - \frac{\delta}{3}$, every $z \in C_{\delta}$ and every $|n| > N_{\epsilon}$ and $i=1,2$ we have:

\begin{equation}\label{goodrot}
\bigg| \frac{f^{n}_{e_i} (z) - z - n\rot_{\nu}(f_{e_i}) }{n} \bigg| < \frac{\epsilon}{4M}.
\end{equation}
Let us choose $n \geq N_{\epsilon}$ such that $|c_{n}| > N_{\epsilon}$ and $|a_{n}| > N_{\epsilon}$ and define $$C := C_{\delta} \cap A^{-n}(C_\delta) \cap A^{-n}f_{e_{2}}^{-c_n}(C_\delta).$$
Observe that $\mu(C) > 1 - \delta$ as we wanted.  Let $x,y \in C$ and define $x'_n := f_{e_{2}}^{c_n}(x'_n) $ and  $y'_n := f_{e_{2}}^{c_n}(y'_n) $. We have by definition of $C$ that the points $x_n, x_n', y_n, y_n' \in C$, and easy computation shows that:
\begin{align*}
f_{B^n(e_1)}(x_n)-x_n = f_{e_1}^{a_n}(x_n')- x'_n + f_{e_2}^{c_n}(x_n)-x_n\\
f_{B^n(e_1)}(y_n)-y_n = f_{e_1}^{a_n}(y_n')- y'_n + f_{e_2}^{c_n}(y_n)-y_n.
\end{align*}
Therefore using  \eqref{goodrot}, the fact that $|a_n|, |c_n| > N_{\epsilon}$ and the triangle inequality we obtain:
\begin{align*}
\frac{|f_{B^n(e_1)}(x_n)-x_n - f_{B^n(e_1)}(y_n)- y_n |}{|\lambda_B^n|} &\leq  \frac{M|f_{e_1}^{a_n}(x_n')- x'_n -  f_{e_1}^{a_n}(y_n')- y'_n|}{|a_n|}\\
&\ \ \ +  \frac{M|f_{e_2}^{c_n}(x_n)-x_n - f_{e_2}^{c_n}(y_n)-y_n|}{|c_n|}\\
&\ < \e/2 + \e/2 = \e.
\end{align*}
Therefore, by equation \eqref{txty} we obtain $|(f_{e_1}(x) -x) - (f_{e_1}(y) - y)| = |t_x - t_y| < \e$ for every $x,y \in C$ as we wanted. We can argue similarly for $f_{e_2}$.

\end{proof}

\begin{proposition}\label{rotzero2} Suppose that $\nu$ is a $\Gamma_B$-invariant probability measure as in  Proposition \ref{nicemeasure}. We have that for $\nu$ a.e. $x \in \T^2$ and every $v \in \Z^2$:  $$f_{v}(x) = x + \rot_{\nu}(f_{v}).$$ That is, $\Z^2$ acts as a translation for $\nu$ a.e .$x \in \T^2$.
\end{proposition}

\begin{proof} We will only prove it in the case $v = e_1, e_2$. 
It follows from Proposition \ref{technical} that the function  $f_{e_{i}}(x) - x$ must be  constant for $\nu$ a.e. $x \in \T^2$ and it follows easily by integration that such constant must be equal to $\rot_{\nu}(f_{e_i})$ finishing the proof of Proposition \ref{rotzero2}. \end{proof}

We are now ready to complete the proof of Theorem \ref{ABC4}.
\begin{proof}[Completing the proof of Theorem \ref{ABC4}]
We will now obtain a contradiction to the non-triviality of the rotation sets and show that $\Z^2$ acts by translations on $\T^2$. Let $\nu$ be a measure as in Proposition \ref{nicemeasure}. Observe that by continuity, Proposition \ref{rotzero2} implies the action of $\Z^2$ restricted to the support $\supp(\nu)$ is by  translations. Moreover we have that $\supp(\nu)$ is also invariant by the Anosov linear transformation $A$ because $\nu$ is $\Gamma_B$-invariant. 

If the matrix $(\rot_{\nu}(f_{e_1}), \rot_{\nu}(f_{e_2}))$ contains an irrational entry we have that $\supp(\nu)$ is invariant by an irrational translation, so $\supp(\nu)$ is either $\T^2$ or contains a circle. In the latter case, the circle intersects the stable manifold of 0 transversally since the eigenvectors of $A$ have irrational slopes. We apply the $\lambda$-lemma to obtain that under iterates the image of the circle will approach the unstable manifold of $0$ and get stretched exponentially so will become dense on $\T^2$. Using the invariance by the Anosov $A$ it follows that $\supp(\nu) = \T^2$. Therefore we conclude that $f_{e_1}, f_{e_2}$ must be translations everywhere.

In the case where $(\rot_{\nu}(f_{e_1}), \rot_{\nu}(f_{e_2}))$ has all entries rational, recall that in the proof of Proposition \ref{nicemeasure} the probability measure $\nu$ was constructed in such a way  that the matrix $(\rot_{\nu}(f_{e_1}), \rot_{\nu}(f_{e_2}))$ is one of the endpoints of the interval $\Sr$. We can also construct in the same way another measure $\nu' \in \T^2$ such that $(\rot_{\nu'}(f_{e_1}), \rot_{\nu'}(f_{e_2}))$ is the other endpoint of $\Sr$, but this matrix cannot have all  entries being rational because otherwise $\Sr$ is a segment with rational slope and this contradicts the fact that $\rot(f_{e_1}), \rot(f_{e_2})$ are segments with an irrational direction given by the unstable eigenvector $\bu$ of $A$. Therefore the matrix $(\rot_{\nu'}(f_{e_1}), \rot_{\nu'}(f_{e_2}))$ contains an irrational entry and in that case $\nu'$ is invariant by an irrational translation and we can argue as before. This give us a contradiction and finishes the proof of Theorem \ref{ABC4}.

\end{proof}

\end{section}

\begin{section}{Case $|\lambda_A|< |\lambda_B|$.}\label{SABC2}

Proposition \ref{ABC4} proves the case of $|\lambda_A|\geq |\lambda_B|$ of Theorem \ref{ABC3}. In this section, we prove Theorem \ref{ABC3} in the case where $|\lambda_B| > |\lambda_A|$ finishing the proof of Theorem \ref{ABC3}, we will use the same notation as in Section \ref{secrota}. We recall that in this case we want to obtain a contradiction from the fact that the action is faithful, we will indeed show that for all $v \in \Z^2$ we have $\Phi(v) = \text{Id}$. From Proposition \ref{difficult}, we can assume that the rotation set for $v \in \Z^2 \subset \Gamma_B$ every diffeomorphisms $f_v$ consists of a single point in $\R^2$ and we will abuse notation and call such vector $\rot(f_v)$. We have from Proposition \ref{conjugarot} that the following equation holds:
$$ \big( \rot(f_{e_1}), \rot(f_{e_2}) \big)B= A\big(\rot(f_{e_1}), \rot(f_{e_2}) \big) + \big( w_1, w_2 \big)$$ for some $w_1, w_2 \in \Z^2$. Therefore from Proposition \ref{LmNull}, as $tr(B) \neq tr(A)$, we have that both $\rot(f_{e_1})$ and $ \rot(f_{e_2})$ are rational vectors. Therefore we can pass to a finite cover of $\Gamma_{B}$ of the form $B \ltimes m\Z^2 $ and assume that the rotation set of every $f_{v}$ is in $\Z^2$ for all $v \in \Z$. Moreover, we will assume that the rotation set of $f_v$  is $\{0\}$ for all $v$, which follows because we can assume that the lifts $f_{e_1}, f_{e_2}$ have rotation set $\{0\}$ and this imply that $f_v$ has rotation set $\{0\}$ for all $v$. We have the following:

\begin{lemma}\label{lifts} The action of $\Gamma_B$ on $\T^2$ lifts to an action of $\Gamma_B$ in the universal covering $\R^2$.
\end{lemma}

\begin{proof}
As we observed, we can take lifts of the $\Z^2$ action in such a way that every $f_v$ has as rotation set $\{0\}$, then one can check that $Af_vA^{-1}$ has rotation set $\{0\}$ as well and conclude that $$Af_vA^{-1} = f_{B(v)}$$ which implies the action $\Gamma_B$ lifts.

\end{proof}

Therefore, we can assume that  $\Gamma_B$ acts on $\R^2$. Moreover, every $f_v$ leaves invariant every leaf of the foliation $\mathcal{F}$, this follows because all $f_v$ have rotation set $\{0\}$, and each $f_v$ acts in the set of leafs as a translation, this is because the lifts of any leaf of $\mathcal{F}$ are dense and as $f_v$ commutes with deck transformations if $f_v(l)=l + \alpha$ for some leaf $l \in \mathcal{F}, \alpha \in \R$, then $f_v(l) = l + \alpha$ for all $l \in \mathcal{L}$, therefore if a leaf of $\mathcal{F}$ is not preserved we will have non-zero rotation set. The following two propositions will be proved in the next subsection.

\begin{proposition}\label{PropFlow}There exists a continuous one parameter subgroup $g_t: \T^2 \to \T^2$ such that
\begin{enumerate}
\item for each $t\in \R$, the map $g_t:\ \T^2\to \T^2$ is measurable;
\item the map $t\mapsto g_t$ induces a group homomorphism $\R \to \text{Homeo}(\R)$ which preserves the unstable foliation of $\Phi(B)$;
 \item we have $g_1 = \Phi(e_1)$ and $g_c = \Phi(e_2)$ where $(1,c)$ is an unstable eigenvector of $B$.
 \end{enumerate}
  Moreover we have $$\Phi(B)g_t\Phi(B)^{-1} = g_{\lambda_Bt}$$ where $\lambda_{B}$ is the larger in norm eigenvalue of $B$.
\end{proposition}
\begin{proposition}\label{LmMeasurable}
The flow $g_t:\ \R\times \T^2\to\T^2$ is differentiable in $t\in \R$ and measurable in $p\in \T^2$. \end{proposition}

We should point out that the smoothness of $g_t(p)$ as function of $t$ will be crucial for the following discussion. Assuming the previous propositions, we can define the vector field $X{(p)} = \frac{dg_t}{dt}|_{t=0} (p)$. Observe that $X(p)$ is a well defined vector field because of \ref{LmMeasurable}, but it might not vary continuously as a function of $p$. From the equality $\Phi(B)g_t\Phi(B)^{-1} = g_{\lambda_Bt}$, we have the following equation:
\begin{equation}\label{eigenvector}
D_p(\Phi(B))X(p) = \lambda_BX(\Phi(B)p), \quad \forall\ p \in \T^2.
\end{equation}
\begin{definition}
\begin{enumerate}
\item Let $$C = \{ p \in \T^2\ |\ \Phi(v)(p) = p,\ \forall\ v \in \Z^2\}.$$ Observe that $C$ is a closed set which is invariant by the action of $\Gamma_B$.

\item For every real $M > 0$, let $$X_M = \left\{p \in \T^2 \ |\ \frac{1}{M} <\|X(p)\| < M \right\}$$ and observe that $\T^2 = C \cup (\bigcup_{M > 0} X_M)$.
\end{enumerate}
\end{definition}
Let $\mu$ be the SRB-ergodic measure preserved by $\Phi(B)$, we have the following:
\begin{proposition}\label{PropLY} If $\mu(X_M) > 0$ for some $M>0$, then we have $|\lambda_A| \geq |\lambda_B|$.
\end{proposition}

\begin{proof} As $\Phi(B)$ is topologically conjugate to $A$, the topological entropy of $\Phi(B)$ is equal to $\log(|\lambda_A|)$. Let $\chi_{\mu}$ be the positive Lyapunov exponent of $\Phi(B)$, if $\mu(X_M) > 0$, then for a.e. $p \in X_M$ we have $\chi_{\mu} = \lim \frac{ \log \|D_p\Phi(B)^n (X(p))\| }{n}$ but from equation \eqref{eigenvector} we have  that $$D_p(\Phi(B)^n)(X(p)) = \lambda_B^nX(\Phi(B)^np)$$ and so Poincar\'e recurrence implies that $\chi_{\mu} = \log(|\lambda_B|)$.

From the Ledrappier-Young entropy formula we have \begin{equation}\label{EqLY}\log(|\lambda_A|) = h_{top}(\Phi(B)) \geq h_{\mu}(\Phi(B)) = \dim_{\mu}(W^u)\log(|\lambda_B|) = \log(|\lambda_B|).\end{equation}

The last equality holds because the dimension $\dim_{\mu}(W^u)$ of the measure $\mu$ along the unstable manifolds is equal to one as $\mu$ is an SRB-measure.

\end{proof}
With this proposition, we complete the proof of Theorem \ref{ABC3} in the case of $|\lambda_B|>|\lambda_A|$.

\begin{proof}[Proof of Theorem \ref{ABC3} in the case of $|\lambda_B|>|\lambda_A|$]
If $\mu(X_M) > 0$ for some $M > 0$ we are done. If no such $M$ exists, $\mu(\bigcup_{M > 0} X_M) = 0$ and so $\mu(C) = 1$, but as the conditional measures of $\mu$ along unstable manifolds are absolutely continuous with respect to Lebesgue by taking a partition subordinate to $\mathcal{F}$ we must have that there exists a closed non-trivial interval $J$ of a leaf of $\mathcal{F}$  such that Lebesgue $a.e.$ point of $J$ is in $C$, but $C$ is closed, and so $J$ must be contained in $C$. Also as $C$ is invariant by $\Phi(B)$, then $\Phi(B^n)(J)$ is contained in $C$, but then as $\bigcup_{n>0} \Phi(B^n)(J)$ is dense we have $C = \T^2$ and we are done.
\end{proof}

\subsection{The reduction to circle maps with Diophantine rotation number}
The remaining work is to prove Proposition \ref{PropFlow} and Proposition \ref{LmMeasurable}. In this section, we first show how the $\Z^2$ action action along each leaf of $\mathcal{W}$ is by translations and induces a circle map with Diophantine rotation number (coming from the group structure of $\Gamma_B$). We then apply the Herman-Yoccoz theory for circle maps to show that the flow $g_t$ is smooth in $t \in \R$. We should also point out that Denjoy's theorem is not enough for concluding the smoothness of $g_t$.

\begin{lemma}\label{ire} For every $v \in \Z^2$, we have $$\lim f_{B^{-n}(v)}(p) = p$$ for all $p \in \R^2$ and moreover such limit converges uniformly with respect to $p$.
\end{lemma}

\begin{proof}
Suppose there is $v \in \Z^2$ and a sequence of points $p_k$ in $\R^2$ such that for some constant $K>0$ and a subsequence $\{n_k\}$  we have $|f_{B^{-n_k}(v)}(p_k) - p_k| \geq K > 0$. This will lead us to a contradiction because then as both $p_k$ and $f_{B^{-n_k}(v)}(p_k)$ are in the same leaf of the foliation $\mathcal{F}$, we have that as $A$ expands the leaves of the foliation, we have that $$\lim |A^{n_k}f_{B^{-n_k}(v)}(p_k) - A^{n_k}(p_k)| = \infty$$

But this is the same as $|f_vA^{n_k}(p_k)- A^{n_k}(p_k)| \to \infty $ and this is not possible because $|f_v(p) - p|$ is bounded independent of $p \in \R^2$ because $f_v$ is a lift of a map in $\T^2$.
\end{proof}

We take a leaf $\ell \in \mathcal{F}$ and consider an open interval $\mathcal{I}_v \subset \ell $ which is invariant under $f_v$ for some $v\in \Z^2\setminus\{0\}$, and where there is no fixed point for the $f_v$. 
If this interval does not exist, that means for any open interval $\mathcal J_v$ that is invariant under $f_v$, there is a fixed point of $f_v$ in the interval, which breaks $\mathcal J_v$ into two subintervals that are invariant under $f_v$ due to the one-dimensionality. Repeating the argument, we get that there is a dense subset of fixed points in $\mathcal J_v$. Then the continuity of $f_v$ implies that each point in the entire interval $\mathcal J_v$ is fixed by $f_v$. Arguing this for any open interval  $\mathcal J_v$ and for all $v\in \Z^2$, we conclude that $\Z^2$ fixes $\ell$. As $\ell$ is dense on the torus, we have that the $\Z^2$ action is trivial. So in the following, we consider the above defined interval $\mathcal I_v$.

We have the following analogue of Kopell's lemma. 

\begin{proposition}
The group $\Z^2$ acts freely in $\mathcal{I}_v$, more precisely, if $f_w(p) = p$ for some $p \in \mathcal I_v$ and some $w(\neq v)\in\Z^2\setminus\{0\}$, then $f_w = \mathrm{id}$.
\end{proposition}

\begin{proof}

If $f_w$ has a fixed point $p$ inside $\mathcal{I}_v$, we can assume that $p$ is an endpoint of an interval $J \subset \mathcal{I}_v$ where $f_v$ is conjugate to a translation on $\R$. Since $f_v$ is known to have no fixed point in $\mathcal I_v$ we get that $v$ and $w$ are linearly independent. We decompose $m B^{-n}(v)=a_nv+b_nw$ for some $m\in \Z$ depending only on $w,v$ and $a_n,b_n\in \Z$.  So we get $$f_{m B^{-n}(v)}p=f_{a_nv}f_{b_n w}p=f_{v}^{a_n}p.$$
For $n$ large, the LHS converges to $p$ by Lemma \ref{ire}, while the RHS moves away from $p$ since $f_v$ is monotone on $\mathcal I_v$ and each point converges to one of the endpoints under iterations of $f_v$. This contradiction verifies our claim.

\end{proof}

It follows that the $\Z^2$ action on any interval like $\mathcal{I}$ is conjugate to a group of translations (Corollary 4.1.4 and Theorem 4.1.37 of \cite{N}), therefore there is a homeomorphism $h\ :\ \mathcal{I} \to \R$ such that $T_v:= hf_vh^{-1}$ is a translation and suppose that $T_{e_1} (p) = p + 1$ and $T_{e_2}(p) = p + c$ for some $c \in \R$.

\begin{lemma}
The vector $(1,c)$ lies in the unstable subspace corresponding to $B$.
\end{lemma}
\begin{proof} Recall that  from Proposition \ref{ire} we have that $T_{B^{-n}(e_i)}(x) \to x$ for $i=1,2$.

Therefore if $$B^{-n} = \begin{pmatrix} a_n &  b_n \\ c_n & d_n \end{pmatrix},$$ we have that $T_B^{-n}(e_1)(p) = p + a_n + cc_n$ and $T_B^{-n}(e_2)(p) = p + b_n + cd_n$ and so $$\begin{pmatrix} a_n &  c_n \\ b_n & d_n \end{pmatrix}\begin{pmatrix} 1 \\ c \end{pmatrix} \to \begin{pmatrix} 0 \\ 0 \end{pmatrix}$$ which implies the fact that $(1,c)$ must be in the unstable subspace corresponding to $B$ (in particular that $c$ is an algebraic irrational number.)
\end{proof}
\begin{lemma}
Let $\mathcal{I}$ be any open interval which is a connected component of $\ell \setminus C$ for some leaf $\ell$ of the unstable foliation $\mathcal{F}$ of the Anosov diffeomorphism $\Phi(B)$. Then restricted to $\mathcal{I}$, we have that $\Phi: \Z^2 \to \Diff(\mathcal{I})$ is conjugate to an action by translations and the conjugacy is smooth.
\end{lemma}
\begin{proof}
It remains to show that the conjugacy is smooth. We reduce the action to a circle map and obtain the smoothness of the conjugacy by the Herman-Yoccoz theory for circle maps.

We have that $\mathcal{I}$ is a $C^{r}$ smooth embedded interval  and observe that $\mathbb S^1 \cong \mathcal{I}/ \Phi(e_1) $ is a smooth circle where $\Phi(e_2)$ acts as a smooth diffeo with rotation number $c$, which is algebraic irrational number and therefore Diophantine ($\forall\ \delta,\ \exists\ \gamma$ such that $|qc-p|>\frac{\gamma}{q^{1+\delta}},\ \forall\ (p,q)\in \Z,\ q\neq 0,$ c.f. Chapter VI, Corollary 1E of \cite{S}) and therefore $\Phi(e_2)$ is conjugate to a rotation with angle $c$ (modulo $1$) by a conjugacy that is $C^{r-1-2\delta}$ (c.f. \cite{KO}).

The action of $\Phi(e_2)$ on $\mathcal{I}/ \Phi(e_1) $ is defined as follows more explicitly. Pick a fundamental domain $D\subset \mathcal{I}$ and $p\in D$ and denote by $\bar D$ the circle obtain from $D$ by identifying two endpoints. Denote by $\bar f: \ \bar D\to \bar D$ the action. Then $$\bar f^np= \Phi(m(n)e_1)\Phi(ne_2)p,$$ where $m(n)$ is the unique integer such that $\Phi(m(n)e_1)\Phi(ne_2)p\in D$. By Herman-Yoccoz theory, there exists a $C^1$ $h:\ \bar D\to\mathbb S^1$ such that $h \bar f h^{-1}p=p+c$ mod $1$. That is to say $$\Phi(m(n)e_1)\Phi(ne_2)p=h^{-1}(h(p)+nc).$$

Define $D_m:=\Phi(me_1)D$ and $\bar f_m:\ \bar D_m\to \bar D_m$ the action on $\bar D_m$. We get that $h_m \bar f_m h^{-1}_m(p)=p+c$ where $h_m:\ \bar D_m\to \mathbb S^1$ is given by $h_m=h \Phi(-me_1)$.

\end{proof}
\subsection{Embedding the action into a flow}
We next embed the $\Z^2$-action into a flow. For a single diffeomorphism on $(0,1)$ without fixed points, this is given by Szekeres theorem (Theorem 4.1.11 of \cite{N}). The construction in our setting is simpler.

\begin{proof}[Proof of Proposition \ref{PropFlow}]
We introduce the following map $\R\to \Phi(\Z^2)$ via $(1,c) \cdot v\mapsto \Phi(v)$. Using the circle map constructed in the previous proof, we give a more explicit description. Choosing $(-1/2,1/2)$ as the fundamental domain for $\R/\Z$ and $\bar D$ as a fundamental domain for $\mathcal I/\Phi(e_1)$. We have a conjugation $h:\ \bar D\to \mathbb S^1$ satisfying
$$\Phi(m(n)e_1)\Phi(ne_2)p=h^{-1}(h(p)+nc)=h^{-1}(h(p)+nc+m(n)).$$
 Let $v\in \Z^2$ be such that $(1,c)\cdot v\in (-1/2,1/2)$ then we have $\Phi(v)p=h^{-1}(h(p)+(1,c)\cdot v)$. This defines the flow for $t$ in the dense subset of $\R$. We then extend the domain of definition of the flow to $\R$ by the smoothness of $h$. Explicitly, we obtain a flow $g_\cdot: \R\to \mathrm{Diff}(\mathcal I)$ via $g_t(p)=h^{-1}(h(p)+t)$ for all $p\in \mathcal I$. 

This gives us a smooth one parameter flow $g_t:\ \mathcal I\to \mathcal I$ without fixed points such that for every $v\in \Z^2$,  $\Phi(v) = g_{t_v}$ for $t_v=(1,c)\cdot v \in \R$. So we have $g_1 = \Phi(e_1)$ and $g_c = \Phi(e_2)$. We next define a one parameter subgroup of maps $g_t: \T^2 \to \T^2$ by defining $g_t$ in $\mathcal{F} \setminus C$ as above and defined $g_t(p) = p$ for all $p \in C$.   By construction $g_t$ is differentiable with respect to $t$, but {\it a priori} $g_t(p)$ is not differentiable in $p$.

From the equation $\Phi(B)\Phi(v)\Phi(B^{-1}) = \Phi(B(v))$ it follows that for $t_v=(1,c)\cdot v$, we have $$\Phi(B)g_{t_v}\Phi(B)^{-1} =\Phi(B)\Phi(v)\Phi(B)^{-1} =\Phi(Bv)=g_{(1,c)\cdot Bv}= g_{\lambda_Bt_v}.$$
Extending the domain of $t_v$, we get
$$\Phi(B)g_{t}\Phi(B)^{-1} = g_{\lambda_Bt}$$ for every $t \in \R$.

\end{proof}
\begin{proof}[Proof of Lemma \ref{LmMeasurable}]
The differentiability in $t$ is seen from the definition of the flow. To see the measurability in $p$, we first notice that $\Phi(v)(p)=g_{t_v}(p)$ where $t_v=(1,c)\cdot v$. This shows that for each $v$, the flow $g_{t_v}(p)$ is differentiable in $p$ for $t_v$ fixed. Choosing $v\in \Z^2$, we get that $t_v$ is defined on a dense subset $\mathcal D$ of $\R$, so this shows that $g_t(p)$ is differentiable in $p$ for $t\in\mathcal D$. For each $t_*\notin \mathcal D$, we choose a sequence $t_n\in \mathcal D $ with $t_n\to t_*$,  we get that $g_{t_n}(p)\to g_{t_*}(p)$ pointwise.  This shows that $g_{t_*}(p)$ is measurable in $p$.
\end{proof}

\section{Higher regularity for translations along leaves}\label{SHigher}
In this section, we prove Theorem \ref{ThmHigher} and \ref{ThmHigher4} of higher regularity. 
\begin{proof}[Proof of Theorem \ref{ThmHigher} ]
We consider the case $\rho= u_A\otimes u_{B^t}$, the other case $\rho= u_{A^{-1}}\otimes u_{B^{-t}}$ is similar. It can be verified that \eqref{EqCommute} is satisfied. Suppose $u_{B}=(1,c)$, where $c$ is a Diophantine irrational number because $B$ is Anosov. This shows that $$h\Phi(e_1)h^{-1}(x)=x+  u_A,\quad h\Phi(e_2)h^{-1}(x)=x+c u_A.$$
Therefore the maps $h\Phi(e_1)h^{-1}$ and $h\Phi(e_2)h^{-1}$ preserve each leaf of the unstable foliation of $A$ and act as translations on each leaf, and the maps $\Phi(e_1)$ and $\Phi(e_2)$ preserves the unstable foliations of $\Phi(B)$. 

In this setting, we recover Proposition \ref{PropFlow}. Therefore we have \eqref{eigenvector} by differentiation with respect to $t$.

Let $\mathrm{vol}$ be the volume measure which is preserved by $\Phi(B)$, by the argument in the proof of Proposition \ref{PropLY} with $\mu$ replaced by $\mathrm{vol}$, the larger Lyapunov exponent of $\Phi(B)$ is $\log\lambda_B$ and we have the same calculation as \eqref{EqLY}. By assumption we have $\lambda_A=\lambda_B$, therefore $\mathrm{vol}$ is the measure of maximal entropy for $\Phi(B)$. On the other hand $\log\lambda_B=\log\lambda_A$ is the entropy of $A$ with respect to the measure $h_*\mathrm{vol}$. We also know that $\log\lambda_A$ is the topological entropy of $A$, so we conclude that $h_*\mathrm{vol}=\mathrm{vol}$, since $\mathrm{vol}$ is the unique measure of maximal entropy for $A$. 
Now we apply Theorem 1.3 of \cite{dL} to complete the proof (see also Theorem F of \cite{SY}). 
\end{proof}

\begin{proof}[Proof of Theorem \ref{ThmHigher4} ]
We first claim:

 {\bf Claim: }{\it  On the affine action side, along each eigen-direction $u_{A,i}$ of $A$ corresponding to the eigenvalue $\lambda_i$ as in the statement of the theorem, there is a $\Z^2$ action whose generators act by translating $cu_{A_i}$ and $c\nu u_{A,i}$, where $\nu$ is a Diophantine number and $c\neq 0$. Therefore, the line along $u_{A,i}$ quotient the $\Z^2$ action gives rise to a circle rotation $x\mapsto x+\nu$ mod 1.} 

\begin{proof}[Proof of the claim]
We cite the following lemma of Katznelson.
\begin{lemma}[\cite{K} Lemma 3] Let $A$ be an irreducible $n\times n$ integer matrix and $v$ be an eigenvector of $A$ corresponding to a real eigenvalue of multiplicity 1. Then $v$ is Diophantine with of type $(C,n-1)$, i.e. there exists $C>0$ such that for all $p\in \Z^n\setminus\{0\}$ we have $|\langle p,v\rangle|\geq \frac{C}{|p|^{n-1}}$.
\end{lemma}
Since we have a $\Z^n$ affine action with rotation matrix $\rho= c u_{A,i}\otimes u_{B^t,i}$, each generator is a translation along $u_{A,i}$ with translation length $c$ times each entry of $u_{B^t,i}$ respectively. If $B$ is irreducible, then by the lemma of Katznelson, the eigenvector $u_{B^t,i}$ is Diophantine of type $(C,n-1)$. Picking any two entries of $u_{B^t,i}$, we get a vector denoted by $c(1,\nu)\in \R^2$ that is Diophantine of type $(C,n-1)$. The claim is then established in this case. Suppose $B^t$ is reducible. Without loss of generality, we consider the case $B^t=\left[\begin{array}{cc}
B_1&D\\
0&B_2\end{array}\right]\in \mathrm{SL}_n(\Z)$ where $B_i,\ i=1,2$ is irreducible, and the cases of more diagonal blocks are similar. Since there is no eigenvalue 1, we have that $B_i\in \Z^{n_i\times n_i}$ satisfies $n_i\geq 2,\ i=1,2.$ If $u_1$ is an eigenvector of $B_1$, then $(u_1,0_{1\times n_2})$ is an eigenvector $B$ with the same eigenvalue. Since $u_1$ is a Diophantine vector of type $(C,n_1-1)$ by the above lemma and has at least two entries, we can pick two entries to obtain a Diophantine vector $c(1,\nu)$. If $u_2$ is an eigenvector of $B_2$ with eigenvalue $\lambda$, then we can solve for $v$ from the equation $(B_1-\lambda Id)v=D u_2$ to get that $(v,u_2)$ is an eigenvector of $B$ with eigenvalue $\lambda$ (note that $\lambda$ is not an eigenvalue of $B_1$ by assumption). Then we pick two entries from $u_2$ to get a Diophantine vector. This completes the proof of the claim.
\end{proof}
We next work on the nonlinear action $\Phi$. We relabel the eigenvalues of $B$ as follows
$$|\lambda^s_{k}|<\ldots<|\lambda^s_{1}|<1<|\lambda^u_{1}|<\ldots<|\lambda^u_{\ell}|,\quad k+\ell=n.$$
It is known that for each $\lambda^{u/s}_i$ there is a one-dimensional foliation tangent to the invariant distribution corresponding to $\lambda^{u/s}_i$ and that the conjugacy $h$ preserves the weakest stable and unstable leaves (Lemma 6.1-6.3 of \cite{Go1}). We then apply the argument in the proof of Theorem \ref{ThmHigher} and the above claim to recover Proposition \ref{PropFlow} and \eqref{eigenvector}, i.e. there is a $\Z^2$ subaction of $\Phi(\Gamma_{B,n})$ along the one dimensional foliation which can be embedded into a flow. This shows that the Lyapunov exponent of $\Phi(B)$ with respect to the volume along the weakest stable and unstable foliations agree with the corresponding Lyapunov exponents for the linear Anosov maps. Applying Theorem F of \cite{SY}, we get that $h$ and $h^{-1}$ are $C^{1+}$ along the weakest stable and weakest unstable foliations. We remark that in the proof of Proposition \ref{PropFlow}, we need to invoke Herman-Yoccoz theory for the regularity of the conjugacy for circle maps induced by the $\Z^2$ subaction. Since the Diophantine vector $(1,\nu)$ in the claim is of type $(C,n-1)$, we get that the conjugacy is $C^{1+}$ if $\Gamma_{B,n}$ acts by $C^{r},\ r>n$ diffeomorphisms.

We next apply Proposition 2.4 of \cite{GKS} to show that $h$ preserves the next weak stable and unstable foliations corresponding to Lyapunov exponents $\lambda^s_2$ and $\lambda^u_2$ respectively. The same argument as the previous paragraph allows us to show that $h$ and $h^{-1}$ are $C^{1+}$ along these foliations.

Repeating the above arguments, we obtain that $h$ and $h^{-1}$ are $C^{1+}$ along each one-dimensional foliation corresponding to each eigenvalue $\lambda_i^{u/s}$. An application of Journ\'e theorem \cite{J} completes the proof.

Note that when we have shown that $n-1$ Lyapunov exponents of $\Phi(B)$ agrees with the corresponding ones of $A$, the remaining one agrees automatically. For this reason, in the statement of the theorem, we may allow $c_i=0$ for $c_i$ corresponding to $\lambda^u_\ell$ or $\lambda^s_k$.

In the three dimensional case, if $\Phi(B)$ is smooth enough, by the main result of \cite{Go2}, the conjugacy can be as smooth as possible.
\end{proof}

\section{Proof of Theorem \ref{ThmHigher2}}\label{SHigher2}
In this section, we give the proof of Theorem \ref{ThmHigher2}.
\begin{proof}[Proof of Theorem \ref{ThmHigher2}]

Since $\Phi(B)$ is Anosov and volume-preserving, we see that the Lebesgue measure $\mathrm{vol}$ is an SRB measure. We next show that $\mathrm{vol}$ is preserved by $h_*$ and is the measure of maximal entropy (see Proposition \ref{PropMME} below), so we get that the entropy $$h_{\mathrm{vol}}(\Phi(B))=\log\lambda_A.$$
 By ergodicity and the Ledrappier-Young entropy formula, the Lyapunov exponent of $\Phi(B)$ with respect to $\mathrm{vol}$ is $\log\lambda_A$.
We then complete the proof of Theorem \ref{ThmHigher2} by applying Theorem 1.3 of \cite{dL} (see also Theorem F of \cite{SY}).
\end{proof}

\begin{proposition} \label{PropMME} Under the assumption of Theorem \ref{ThmHigher2}, we have that $h_*\mathrm{vol}=\mathrm{vol}$, and $\mathrm{vol}$ is the measure of maximal entropy of $\Phi(B)$.
\end{proposition}

\begin{proof}

Since each $\Phi(v)$ is volume-preserving by assumption, we get that $h_*\Phi(v)_*\mathrm{vol}=h_*\mathrm{vol}$. On the other hand, using the conjugacy $h$, we can express each $\Phi(v),\ v\in \Z^2,$ as
\begin{equation}\label{EqTrans}
\Phi(v)(x)=h^{-1}(h(x)+\rho v).\end{equation}
Therefore  $$h_*\Phi(v)_*\mathrm{vol}=(h_*\Phi(v)_*h_*^{-1})(h_*\mathrm{vol})=(h\Phi(v)h^{-1})_*(h_*\mathrm{vol})=(h_*\mathrm{vol})(\bullet+ \rho v).$$
So each $h_*\Phi(v)_*\mathrm{vol}$ is a translation of $h_*\mathrm{vol}$ by $\rho v$ and we have
$$h_*\mathrm{vol}=h_*\mathrm{vol}(\bullet+\rho v),\quad \forall \ v\in \Z^2.$$ 

Proposition 1.2 of \cite{WX} shows that the affine action of the form \eqref{EqAffine} is faithful if and only if $A$ is of infinite order and $\{\rho p\ |\ p\in \Z^2\}$ is dense on $\T^2$. So the faithfulness assumption in the statement implies that $h_*\mathrm{vol}$ is invariant under the group of translations on $\T^2$, therefore we get  $h_*\mathrm{vol}=\mathrm{vol}$. Moreover, since vol is the measure of maximal entropy for $h\Phi(B)h^{-1}=A$ and the measure of maximal entropy is preserved under conjugacy, we see that $\mathrm{vol}$ is the measure of maximal entropy for $\Phi(B)$.

\end{proof}

\section{Proof of Theorem \ref{ThmHigher3}}\label{SHigher3}
In this section, we prove Theorem \ref{ThmHigher3}.
\begin{proof}[Proof of Theorem \ref{ThmHigher3}]
Let $\mu$ be an ergodic SRB measure of $\Phi(B)$. We consider the following measure
\begin{equation}\label{Eqnu}\nu:=\lim_{N\to \infty}\frac{1}{N}\sum_{n=1}^N\Phi(B^n(e_1))_*\mu.\end{equation}
We show in the following Proposition \ref{PropSRBMME} that $\nu=h^{-1}_*\mathrm{vol}$ is well-defined and is both an SRB measure and the measure of maximal entropy of $\Phi(B)$.  By the Ledrappier-Young entropy formula, since $\nu$ is the measure of maximal entropy,  the positive Lyapunov exponent of $\Phi(B)$ with respect to $\nu$ is $\log\lambda_B$. Applying Theorem F of \cite{SY} we conclude that the conjugacy $h$ restricted to the support of $\nu$ is $C^{r-\eps}$ along each  unstable leaf in supp$\nu$, and $h^{-1}$ is $C^{r-\eps}$ along each  unstable leaf of $A$ in $\mathrm{supp}h_*\nu=\T^2$, for all $\eps$.

Next, by Corollary 11.14 of \cite{BDV}, we know that $\nu$ as an SRB measure is a Gibbs $u$-state and its support contains entire leaves. Since $\Phi(B)$ is Anosov on $\T^2$, an entire leaf of $\Phi(B)$ is dense on $\T^2$, so we conclude that supp$\nu$ is the whole $\T^2$. This shows that the conjugacy $h$ and $h^{-1}$ are $C^{r-\eps}$ along each unstable leaf.

Applying the same argument to $\Phi(B)^{-1}$, by taking an SRB measure $\mu'$ of  $\Phi(B)^{-1}$, the averaging procedure $\lim_{N\to \infty}\frac{1}{N}\sum_{n=1}^N\Phi(B^{-n} e_1)_*\mu'$ yields the same measure $\nu=h^{-1}_*\mathrm{vol}$ (see the proof of Proposition \ref{PropSRBMME}). Repeating the above argument, we conclude that the conjugacy $h$ and $h^{-1}$ are $C^{r-\eps}$ along each stable leaf   for all $\eps$. Then by Journ\'e theorem \cite{J}, we get that $h$ and $h^{-1}$ are both $C^{r-\eps}$ for all $\eps$.
\end{proof}
\begin{proposition}\label{PropSRBMME}
The limit in the definition \eqref{Eqnu} of $\nu$ exists, and $\nu=h^{-1}_*\mathrm{vol}$ is an SRB measure and the measure of maximal entropy for $\Phi(B)$.
\end{proposition}
\begin{proof}
Consider the following sequence $$h_* \frac{1}{N}\sum_{n=1}^N\Phi(B^n(e_1))_*\mu= \frac{1}{N}\sum_{n=1}^N h_*\Phi(B^n(e_1))_*\mu.$$

Using \eqref{EqTrans}, the following holds $$h_*\Phi(B^n(e_1))_*\mu=(h\Phi(B^n(e_1))h^{-1})_* (h_*\mu)=h_*\mu(\bullet+ \rho B^ne_1). $$

We know from \eqref{EqRotMat} that
$$\rho =\rho_C+ c_1 u_A\otimes u_{B^t}+c_2 u_{A^{-1}}\otimes u_{B^{-t}},$$
where $\rho_C$ solves $A\rho=\rho B+C,\ C\in \Z^{2\times 2}$. Then we have $A^n\rho=\rho B^n+C_n$ with $C_n\in \Z^{2\times 2}.$

Then  we have $$\rho B^{n}e_1=A^n\rho e_1-C_ne_1= c_1(u_{B^t}\cdot e_1) \lambda_B^{n} u_A+c_2(u_{B^{-t}}\cdot e_1)\lambda_B^{-n} u_{A^{-1}}\ \mathrm{mod}\ \Z^2.$$
Therefore, the map $\rho e_1\mapsto \rho Be_1$ on $\T^2$ can be equivalently considered as the map $\rho_c\mapsto A\rho_c$ with $$\rho_c=c_1(u_{B^t}\cdot e_1) u_A+c_2(u_{B^{-t}}\cdot e_1)  u_{A^{-1}}.$$ Note that linear Anosov is ergodic on $\T^2$ with respect to Lebesgue. We thus obtain the full measure set $\mathcal C$ of parameters as the set of $(c_1,c_2)$ such that the corresponding $\rho_c$ is a generic point for applying Birkhoff ergodic theorem to $A$.

Therefore by the Ergodic Theorem we get $\lim\frac{1}{N}\sum_{n=1}^N\delta_{x+A^n\rho_c}=\mathrm{vol}$,  where $\delta_x$ is the Dirac-$\delta$ supported at $x$, also the limit is independent of $x$ by the translation invariance of vol. Therefore we have
\begin{equation*}
\begin{aligned}
\lim\frac{1}{N}\sum_{n=1}^N h_*\Phi(B^n(e_1))_*\mu&=\lim \frac{1}{N}\sum_{n=1}^N h_*\mu(\bullet +A^n\rho_c)\\
&=\int\lim \frac{1}{N}\sum_{n=1}^N\delta_{x+A^n\rho_c}dh_*\mu(x)\\
&=\mathrm{vol}.
\end{aligned}
\end{equation*}

This shows that $\nu=h^{-1}_*\mathrm{vol}$ is well-defined and is invariant under $\Phi(B)$. Since vol is the measure of maximal entropy (MME) for $A$ and the MME is preserved under conjugacy, we conclude that $\nu$ is the MME for $\Phi(B)$.

We next show that $\nu$ is an SRB measure. First, by Corollary 11.14 of \cite{BDV}, we know that the conditional measure of $\mu$ along each unstable leaf in supp$\mu$ is absolutely continuous with respect to Legesgue with a density $\rho$ satisfying $$\frac{1}{K} \leq \frac{\rho(z_1)}{\rho(z_2)}\leq K$$ for some constant $K$ uniform for all leaves in supp$\mu$ and all $z_1,z_2$ in the same leaf with $d(z_1,z_2)<10$ (equation (11.4) of \cite{BDV}). Next, by the following Lemma \ref{LmDerivativeBound}, we know that $\|D^u_x\Phi(B^n(e_1))\|$ along each unstable leaf is bounded from above and below uniformly with respect to both $n\in \N$ and $x\in \T^2$. This shows that $\Phi(B^n(e_1))_*\mu$ has conditional measure absolutely continuous with respect to Lebesgue with a density $\rho_n$ satisfying $$\frac{1}{K'} \leq  \frac{\rho_n(z_1)}{\rho_n(z_2)}\leq K' $$ where the constant $K'$ is uniform for all leaves in supp$\Phi(B^n(e_1))_*\mu$, all $z_1,z_2$ in the same leaf with $d(z_1,z_2)<10K'$ and for all $n\in \N$. The averaged measure $\nu$ therefore has the same property hence is SRB.
\end{proof}
The proof is then reduced to establishing a uniform derivative bound of $\Phi(B^n(e_1))$ along the unstable leaves.

\begin{lemma}\label{LmDerivativeBound}
For each $v\in \Z^2$, there is a constant $K=K_v$ such that
$$1/K\leq \|D^u_x\Phi(B^n(v))\|\leq K,\quad 1/K\leq \|D^s_x\Phi(B^{-n}(v))\|\leq K,\quad \forall\ x\ \in \T^2,\ n\in \N.$$
\end{lemma}
\begin{proof}We prove the statement for $D^s\Phi(B^{-n}(v))$. The other one is analogous.

Note that $h$ is biH\"older. Denote by $\alpha$ the H\"older exponent.
Fix a small $r$, then we have
$$B_{r^{1/\alpha}}(h(x))\subset hB_r(x)\subset B_{r^{\alpha}}(h(x)).$$

Choose $z\in B_r(x)\cap W^s_x$ such that $d(\Phi(B^{-n}(v))x,\Phi(B^{-n}(v))z)=\delta$ for some $\delta>0$. Since $h\Phi(B^{-n}(v)) h^{-1}$ acts as translation, by the H\"olderness of $h$, we get $\delta^{1/\alpha}\leq d(x,z)\leq \delta^{\alpha}$. Next, by the mean value theorem, we have for some $y$ lying between $x$ and $z$ on the same leaf
$$\delta=d(\Phi(B^{-n}(v))x,\Phi(B^{-n}(v))z)=  d(x,z)\cdot \|D^s_y\Phi(B^{-n}(v))\|$$
Next, applying the following Lemma \ref{LmDistortion} we have
 $$1/C \delta^{1-\alpha}\leq \|D^s_x\Phi(B^{-n}(v))\|\leq C\delta^{1-1/\alpha}.$$
\end{proof}

\begin{lemma}\label{LmDistortion}
For each $v\in \Z^2$, there exists constant $C=C_v>1$ such that for any two nearby points $x$ and $y$ on the same unstable leaf, we have
$$\left|\log\frac{D^u_x \Phi(B^n(v))}{D^u_y \Phi(B^n(v))}\right|\leq C|x-y|^{\alpha^2}$$
for all $n\in \N$, where $\alpha$ is the H\"older constant of the conjugacy $h$ and $h^{-1}$.

Similarly, we have for all nearby $x,y$ on the same stable leaf $$\left|\log\frac{D^s_x \Phi(B^{-n}(v))}{D^s_y \Phi(B^{-n}(v))}\right|\leq C|x-y|^{\alpha^2}.$$
for all $n\in \N$.
\end{lemma}
\begin{proof}

By the group relation $ \Phi(B)^n\Phi(v)\Phi(B)^{-n}=\Phi(B^{n}(v))$, we get
 $$ D_{\Phi(v)x_{-n}}\Phi(B)^nD_{x_{-n}}\Phi(v)D_x\Phi(B)^{-n}=D_x\Phi(B^{n}(v)),$$
 where we use the notation $x_i$ to denote the $i$-th orbit point under the iteration $\Phi(B)$. In the following we denote $x'_n:=\Phi(v)x_{-n}$ and $y'_n:=\Phi(v)y_{-n}$ and by $(x'_n)_i$ the orbit point of $x'_n$ under $\Phi(B)$, similarly for $y'_n$.

 By assumption, $x$ and $y$ lie on the same unstable leaf of $\Phi(B)$. So we get that $|x_{-i}-y_{-i}|<\lambda^i|x-y|$ for some $0<\lambda<1$.

 We have the following estimate
 \begin{equation}\label{calcu}
 \begin{aligned}
 \log\frac{D^u \Phi(B^n(v))(x)}{D^u \Phi(B^n(v))(y)}&=\log \frac{D^u_{x_{-n}}\Phi(v)}{D^u_{y_{-n}}\Phi(v)}-\sum_{i=0}^{n-1}\log \frac{D^u_{x_{-i}}\Phi(B^{-1})}{D^u_{y_{-i}}\Phi(B^{-1})}+\sum_{i=0}^{n-1}\log \frac{D^u_{(x'_{-n})_{i}}\Phi(B)}{D^u_{(y'_{-n})_{i}}\Phi(B)}\\
 &\leq \frac{(D^u_{y_{-n}})^2\Phi(v)}{D^u_{y_{-n}}\Phi(v)}|x_{-n}-y_{-n}|\\
 &+\sum_{i=0}^{n-1}\frac{(D^u_{y_{-i}})^2\Phi(B)}{D^u_{y_{-i}}\Phi(B)}|x_{-i}-y_{-i}|\\
 &+\sum_{i=0}^{n-1}\frac{(D^u_{(y'_{-n})_{i}})^2\Phi(B)}{D_{(y'_{-n})_{i}}\Phi(B)}|(x'_{-n})_{i}-(y'_{-n})_{i}|.
 \end{aligned}
 \end{equation}
 Now by the exponential expansion and contraction, the estimate of the second summand on the RHS is bounded  by the distance $C\|\frac{(D^u)^2\Phi(B)}{D^u\Phi(B)}\| |x-y|$.  Also, each term of the third summand on the RHS is bounded by $C\|\frac{(D^u)^2\Phi(B)}{D^u\Phi(B)}\| |(x'_{-n})_i-(y'_{-n})_i|$.\\
 As the conjugacy $h$, which is known to be bi-H\"older, we have that  $$\frac{1}{C}|x-y|^{1/{\alpha}} \leq |h(x)-h(y)| \leq C|x-y|^\alpha. $$
 As $$|h((x'_{-n})_i)- h((y'_{-n})_i)| = |A^{i-n} h(x) - A^{i-n} h(y)|,$$ one has that  $$|(x'_{-n})_i-(y'_{-n})_i| \leq C'|x-y|^{\alpha^2}\lambda^{\alpha (n-i)}$$ and so the third summand in (\ref{calcu}) is bounded by $C''|x-y|^{\alpha^2}$.
\end{proof}

\section*{Acknowledgment}
We would like to thank Professor A. Wilkinson for getting us interested in the problem and many stimulating discussions. We also would like to thank the referees for pointing out a mistake (corrected in Section 7) in the first version. J. X. is supported by NSFC (Significant project No.11790273) in China and Beijing Natural Science Foundation (Z180003).

\end{section}

\end{document}